\documentclass[11pt]{article}
    
\usepackage{amsmath,amsfonts,amsthm,amscd,amssymb,graphicx}
\usepackage{mathabx}

\usepackage[utf8]{inputenc}

\usepackage{subfigure}
\usepackage{xcolor}

\numberwithin{equation}{section}

\usepackage{hyperref}

\newtheorem{theorem}{Theorem}[section]
\newtheorem{theo}{Theorem}[section]
\newtheorem{lemma}[theorem]{Lemma}
\newtheorem{lem}[theorem]{Lemma}
\newtheorem{proposition}[theorem]{Proposition}

\newtheorem{corollary}[theorem]{Corollary}

\newtheorem{remark}[theorem]{Remark}

\newtheorem{defi}[theorem]{Definition}

\newcommand{\RR}{\mathbb{R}}
\newcommand{\cO}{\mathcal{O}}

\newcommand{\cE}{\mathcal{E}}

\newcommand{\ZZ}{{\mathbb Z}}
\newcommand{\FF}{{\widehat F}}

\def\hv{\hat{v}}

\def\hV{\widehat{V}}
\def\hPsi{\widehat{\Psi}}

\def\Fg{\widehat{g}}

\def\Frho{\widehat{\rho}}
\def\Fphi{\widehat{\phi}}
\def\FE{\widehat{E}}

\def\FG{\widehat{G}}

\def\FB{\widehat{B}}

\def\Trho{\widetilde{\rho}}

\def\Tphi{\widetilde{\phi}}

\begin{document}

\title{A new framework for particle-wave interaction} 

\author{Toan T. Nguyen\footnotemark[1]
}

\maketitle

\footnotetext[1]{Penn State University, Department of Mathematics, State College, PA 16802. Email: nguyen@math.psu.edu. The research is supported in part by the NSF under grants DMS-2054726 and DMS-2349981. The author would like to acknowledge the hospitality of the Institut des Hautes \'Etudes Scientifiques for a visit in Spring 2024 during which part of this research was
carried out.
}

\begin{abstract}
In plasma physics, collisionless charged particles are transported following the dynamics of a meanfield Vlasov equation with a self-consistent electric field generated by the charge density. Due to the long range interaction between particles, the generating electric field oscillates and disperses like a Klein-Gordon dispersive wave, known in the physical literature as plasma oscillations or Langmuir's oscillatory waves. The oscillatory electric field then in turn drives particles. Despite its great physical importance, the question of whether such a nonlinear particle-wave interaction would remain regular globally and be damped in the large time has been an outstanding open problem. In this paper, we propose a new framework to resolve this exact nonlinear interaction. Specifically, we employ the framework to establish the large time behavior and scattering of solutions to the nonlinear Vlasov-Klein-Gordon system in the small initial data regime. The novelty of this work is to provide a detailed physical space description of particles moving in an oscillatory field and to resolve oscillations for the electric field generated by the collective interacting particles. This appears to be the first such a result analyzing oscillations in the physical phase space $\RR^3_x\times \RR_v^3$. 

\end{abstract}


\tableofcontents



\section{Introduction}


One of the most central questions in plasma physics is whether a plasma, namely a collection of charged particles, in a non-equilibrium state will transition to turbulence or relax to neutrality in the large time. The question poses a great mathematical challenge due to the extremely rich underlying physics including phase mixing, Landau damping, oscillations, and trapped geodesics, among others. While {\em phase mixing} is purely a free transport phenomenon, {\em plasma oscillations} arise self-consistently due to the long-range pair interaction between the charged particles that generates oscillatory electric fields, also known as Langmuir's oscillatory waves in the physical literature \cite{Trivelpiece}. These Langmuir waves oscillate and disperse like a Klein-Gordon dispersive wave as recently confirmed in the mathematical literature \cite{HKNR3, BMM-lin, Toan, Ionescu1, HKNR4} for a linearized system near fixed background equilibria. 

L. Landau in his original paper \cite{Landau} addresses the very question of whether such an oscillation is damped (i.e. energy exchange from waves to particles, or potential to kinetic energy), ever since known as {\em Landau damping}, for the linearized system. Landau's law of decay can be explicitly computed, as done in \cite{Landau} near Gaussian states, and is sensitive to the decaying rate of the background electrons: the faster the background vanishes at its maximal velocity, the weaker Landau damping is. Furthermore, as was recently discovered in \cite{Toan, HKNR4}, there is a nontrivial {\em survival threshold} of wave numbers that completely characterizes the decay mechanism for the linearized electric field: phase mixing above the threshold, Landau damping at the threshold, and the survival of pure Klein-Gordon oscillations below the threshold; see also \cite{ChanjinToan} for a quantum mechanical counterpart. This dynamical picture remains notoriously open for the full nonlinear problem. In this paper, we resolve the nonlinear particle-wave interaction issue, namely resolving (non) {\em Landau damping below survival threshold.} 

Specifically, in this paper, we consider the following relativistic Vlasov-Klein-Gordon system
\begin{equation}
\label{VKG1}
\partial_t f + \hv\cdot \nabla_x f + E\cdot \nabla_v f = 0 
\end{equation} 
\begin{equation}\label{VKG2}
E = -\nabla_x \phi, \qquad (\Box_{t,x}+1) \phi = -\rho[f] ~~~ 
\end{equation}
with $\Box_{t,x} = \partial_t^2 - \Delta_x$, modeling the dynamics of an electron moving in a self-consistent oscillatory field, which resembles the exact particle-wave interaction encountered below the survival threshold. The particles are transported with relativistic velocity $ \hv = v/\langle v\rangle$, with $\langle v\rangle=\sqrt{1+|v|^2}$, while the electric field $E(t,x)$ is self-consistent and generated by the charge density 
$$\rho[f](t,x) = \int_{\RR^3} f(t,x,v)\; dv$$ 
through the Klein-Gordon equation. The system \eqref{VKG1}-\eqref{VKG2} is nonlinear through the quadratic interaction $E\cdot \nabla_v f$, since $E$ depends linearly on $f$ through its density $\rho[f](t,x)$. The system will be solved for initial data of the form 
\begin{equation}\label{VKG3}f(0,x,v) = f_0(x,v), \qquad \phi(0,x) = \phi_0(x), \qquad \partial_t\phi(0,x) = \phi_1(x). \end{equation}

The system \eqref{VKG1}-\eqref{VKG3} may also be seen as a special case of the relativistic Vlasov-Maxwell system where the electromagnetic fields are self-consistently generated through the classical Maxwell equations, and has been studied in the literature \cite{Kunzinger1, Kunzinger2}. However, as is the case of Vlasov-Maxwell systems \cite{GlasseyStrauss, Klainerman, BGP}, global smooth solutions to the Vlasov-Klein-Gordon system for general data are not known to exist. Most intriguingly, the question of global existence and scattering theory is widely open for the Vlasov-Klein-Gordon system even for small initial data. Namely, there are no known global in time classical solutions to \eqref{VKG1}-\eqref{VKG3}, but global weak solutions or continuation criteria for classical solutions \cite{Kunzinger1, Kunzinger2}. This is in great contrast to the case of Vlasov-Maxwell systems where the small data regime is well-understood  \cite{GlasseyStrauss2, Big1, Big2, Wang}. See also \cite{Calogero1, Calogero2} for a related Vlasov-{N}ordstr\"{o}m system describing the self-gravitating collisionless matter.  

In this paper, we establish the global well-posedness and scattering theory in the small data regime, therefore resolving the particle-wave interaction at the nonlinear level. The novelty of this work is to provide a new framework to study oscillations in the physical phase space and nonlinear particle-wave interaction via a detailed description of particles moving in an oscillatory field and of plasma oscillations generated by the collective interacting particles. The framework of dealing with the interaction between particles and oscillatory waves should also be found useful in several other contexts such as Vlasov-Maxwell, Vlasov-{N}ordstr\"{o}m, or Einstein-Vlasov systems. We also mention \cite{BD, IPWW1, GlasseyStrauss2, Big1, BMM-cpam, HKNR2, TrinhLD, IPWW} for a few related works
on the large time behavior and scattering of solutions to collisionless Vlasov models in the whole space $\RR^3_x\times \RR_v^3$.  

\subsection{Main result}

Our main result reads as follows. 

\begin{theo} \label{maintheorem} Fix $\alpha_0\ge 8$, and let $f_0(x,v), \phi_0(x), \phi_1(x)$ be initial data of the relativistic Vlasov-Klein-Gordon system \eqref{VKG1}-\eqref{VKG3}. Suppose that $f_0(x,v)$ is compactly supported in $v$, and in addition,  
\begin{equation}\label{data-assumptions}
\| \phi_0\|_{H^{\alpha_0+2}_x} + \| \phi_1\|_{H^{\alpha_0+1}_x}
+ \| f_0\|_{H^{\alpha_0+1}_{x,v}} \le \epsilon_0.
\end{equation}
Then, for sufficiently small $\epsilon_0$, the solution $(f(t,x,v), E(t,x))$ to the relativistic Vlasov-Klein-Gordon system \eqref{VKG1}-\eqref{VKG3} exists globally in time, and the nonlinear electric field $E(t,x)$ remains oscillatory and is of the form 
\begin{equation}\label{main-repE}
\begin{aligned}
E &= \sum_\pm E^{osc}_\pm(t,x) + E^r(t,x),
\end{aligned}\end{equation}
where $E^{osc}_\pm(t,x)$ behaves as a pure Klein-Gordon wave, and $E^r(t,x)$ denotes a remainder term. In particular, there hold
\begin{equation}\label{maintheo-bdsE}
\begin{aligned}
\| E^{osc}_\pm(t)\|_{L^p_x} &\lesssim \epsilon_0 \langle t\rangle^{-3(1/2-1/p)} , \qquad p\in [2,\infty],
\\
\| E^r(t)\|_{L^p_x} &\lesssim \epsilon_0 \langle t\rangle^{-3(1-1/p)} , \qquad p\in [1,\infty],
\end{aligned}
\end{equation}
for all $t\ge 0$. 
In addition, particles scatter in the large time, namely for any initial state $(x,v)$, there exists a final velocity $V_\infty(x,v)$ in $C^1(\RR^3\times \RR^3)$ so that the particle trajectories\footnote{Throughout this paper, we in fact work with the backward characteristic, see Section \ref{sec-char}.} $(X(t;x,v), V(t;x,v))$ satisfy 
\begin{equation}\label{scatterXV}
\| V(t;x,v) - V_\infty(x,v)\|_{L^\infty_{x,v}}  \lesssim \epsilon_0 \langle t\rangle^{-3/2}, \qquad \| X(t;x,v) - x - tV_\infty(x,v)\|_{L^\infty_{x,v}} \lesssim \epsilon \langle t\rangle^{-1/2}
.\end{equation}
In particular, there exists an $f_\infty(x,v)$ in $C^1(\RR^3\times \RR^3)$ so that 
\begin{equation}\label{scatterf}\|f(t,x+tV_\infty(x,v),V_\infty(x,v)) - f_\infty(x, v)\|_{L^\infty_{x,v}} \lesssim \epsilon_0\langle t\rangle^{-1/2}, \end{equation}
for all $t\ge 0$. 
\end{theo}

Let us comment on the main results. First, the fact that the electric field remains oscillatory may be seen from the Klein-Gordon equation $(\Box_{t,x}+1)E = \nabla_x \rho$, which may formally be written in the form 
\begin{equation}\label{introE} E(t,x) = \sum_\pm G^{osc}_\pm \star_{t,x} \nabla_x \rho(t,x) ,\end{equation}
plus initial data contributions, where $G^{osc}_\pm(t,x)$ denotes the Klein-Gordon's Green function (namely, the Fourier inverse of the spatial symbol $\langle k\rangle^{-1}e^{\pm i \langle k\rangle t}$). Due to the transport dynamics, we expect that $\rho(t,x)$ decays at an order of $\langle t\rangle^{-3}$, however unlike phase mixing (e.g., \cite{BD,BMM-cpam,HKNR2, TrinhLD, Big1}), its spatial derivatives may not gain any extra decay due to the presence of oscillations, which is one of the main issues in dealing with Klein-Gordon's type dispersion. The decay of $\rho(t,x)$ is therefore insufficient to derive that of $E(t,x)$ through the spacetime convolution \eqref{introE}. The fact that the decay is insufficient can also be seen directly from the nonlinear interaction $E \cdot \nabla_v f$, which is formally of order $t^{-1/2}$, since $\partial_v = \mathcal{O}(t)$ due to the free transport, which is far from being integrable in time, not to mention the apparent loss of derivatives in $v$. 

One of the novelties in the present approach, similar to \cite{HKNR2}, is to work in the Lagrangian coordinates, namely the particle trajectories, for which the apparent loss of derivatives in $v$ is avoided, following the transport characteristic of \eqref{VKG1}:
\begin{equation}\label{ode} \dot x = \hv, \qquad \dot v = E(t,x). \end{equation}
However, the lack of decay remains a serious issue in locating position of the particles, as is being the double integration in time of the electric field $E(t) = \mathcal{O}(t^{-3/2})$. One of the key observations in this work to overcome such a lack of decay is that time integration of the oscillatory field along the particle trajectories is better than expected. Indeed, suppose that the electric field is of the form $E_\pm^{osc}(t,x) = e^{ik\cdot x + \lambda_\pm(k) t}$, with $\lambda_\pm(k) = \pm i\langle k\rangle$ (namely, a Klein-Gordon wave at frequency $k$). Setting $\omega^v_\pm(k) = \lambda_\pm(k) + ik\cdot \hv$,   
we compute 
\begin{equation}\label{keyint}
 \begin{aligned}
 \int_0^t E_\pm^{osc}(\tau, x+\hv \tau)\; d\tau 
 &=  \int_0^t e^{ik\cdot x + \omega^v_\pm(k)\tau}\; d\tau 
= \frac{1}{\omega^v_\pm(k)} (e^{ik\cdot x + \omega^v_\pm(k)t} - e^{ik\cdot x})
\\& = \frac{1}{\omega^v_\pm(k)} \Big(E^{osc}_\pm(t,x+\hv t) - E^{osc}_\pm(0,x)\Big).
 \end{aligned}
\end{equation}
Effectively, up to a shift of $x \mapsto x-\hv t$, this calculation shows that particle velocities are a superposition of a purely oscillatory component and a pure transport part. This very decomposition turns out to hold for the genuine nonlinear particle trajectories where the electric field is nonlinear and of the form \eqref{introE}, see Proposition \ref{prop-charV}. A similar decomposition is also carried out for particle positions, see \eqref{prop-charPsi}. The precise description of particle positions and velocities in the physical space is one of the novelties of this work.  

Next, to overcome the lack of decay in the spacetime convolution \eqref{introE} and to propagate the Klein-Gordon dispersion nonlinearly,  we again perform the time integration at the level of nonlinear interaction, remarkably leading to 
\begin{equation}\label{keyint2}
\begin{aligned}
G^{osc}_\pm \star_{t,x} \nabla_x \rho(t,x)  &= \frac{1}{\omega^v_\pm(i\partial_x)}G^{osc}_\pm(t) \star_{x} \nabla_x \rho(0,x) - \frac{1}{\omega^v_\pm(i\partial_x)}G^{osc}_\pm(0) \star_{x} \nabla_x \rho(t,x)
\\&\quad + \frac{1}{\omega^v_\pm(i\partial_x)} G^{osc}_\pm \star_{t,x} \nabla_x [\rho E](t,x) .
\end{aligned}
\end{equation}
See Proposition \ref{prop-convGoscS0} for the precise details. The presentation \eqref{keyint2} reveals the deep structure hidden in the nonlinear interaction. Namely, the nonlinear electric field is again a superposition of a pure oscillation (i.e. Klein-Gordon dispersion) and a pure transport (i.e. phase mixing), plus a higher-order nonlinearity. This decoupling of oscillations from phase mixing is the key to the nonlinear iteration.  
  
Finally, it is well-known that the spacetime convolution \eqref{introE} experiences loss of derivatives (e.g., \cite{RS, HKNR4}), see Lemma \ref{lem-decayoscS0}. We may attempt to resolve the issue, following the ``standard procedure'', by deriving decay for low norms and propagating high norms with some possible growth in time. However, this does not work for Vlasov equations, since at the top order of derivatives, the nonlinear interaction $E_\mathrm{low} \cdot \nabla_v f_\mathrm{high} = \mathcal{O}(t^{-1/2})$ and $E_\mathrm{high} \cdot \nabla_v f_\mathrm{low} = \mathcal{O}(1)$, leading to a growth in time of orders $\langle t\rangle^{1/2}$ and $\langle t\rangle$, respectively (noting $f$ does not decay and $E$ has no oscillation in high norms). As should be clear by now, the novelty of this work is to completely avoid dealing with the Vlasov equation but to directly work with the densities (and hence accurately tracing the genuine nonlinear characteristic), see Section \ref{sec-sourceest}. This enables us to propagate phase mixing up to the top derivatives, remarkably allowing a growth in time not only at the top order, but also all but one derivatives of the characteristic (see Proposition \ref{prop-HDchar}). The nonlinear iterative scheme with a cascading decay in norms is devised in Section \ref{sec-bootstrap}. The cascade of decay does not come from the simple interpolation between low and high norms, but from the phase mixing estimates for the transport (i.e. the contracting in volume due to shearing), see \eqref{bootstrap-decaydaS} and Proposition \ref{prop-bdS}. 

 The paper is outlined as follows. Section \ref{sec-nonlinearframework} is devoted to highlighting the new Lagrangian framework and the nonlinear iterative scheme to keep track of oscillations. Section \ref{sec-Char} provides the precise description of the nonlinear characteristic, where oscillations are detailed at the particle level. Section \ref{sec-sourceest} treats the source density, namely the density generated by the nonlinear characteristics, leaving Section \ref{sec-decayosc} to exploit the crucial hidden structure of nonlinear interaction between transported particles and oscillatory waves.  


\section{Nonlinear framework}\label{sec-nonlinearframework}



\subsection{Characteristics}\label{sec-char}


Following \cite{HKNR2}, we shall solve the nonlinear Vlasov-Klein-Gordon system via the method of characteristics. 
Precisely, let us introduce the characteristics $(X_{s,t}(x,v),V_{s,t}(x,v))$, which are defined by 
\begin{equation}\label{ode-char} 
\frac{d}{ds}X_{s,t} = \hV_{s,t}, \qquad \frac{d}{ds}V_{s,t} = E(s,X_{s,t})
 \end{equation}
with initial data 
$$
X_{t,t} (x,v)= x, \qquad V_{t,t}(x,v)= v,
$$ 
where we recall the notation for the relativistic velocity 
$\hv  = v/\sqrt{1+|v|^2}$. 
Throughout this paper, we work with the characteristics ``backwards'', namely for $s \le t$.
By definition, we have 
\begin{equation} \label{ode-char2}
\begin{aligned}
X_{s,t} (x,v) &= x - \int_s^t \hV_{\tau,t}(x,v) \; d\tau , 
\\
 V_{s,t}(x,v) &= v - \int_s^t E(\tau, X_{\tau,t}(x,v)) \; d\tau.
\end{aligned}
\end{equation}
Integrating \eqref{VKG1} along characteristics, we obtain an explicit formula for the solution, namely 
\begin{equation}
\label{charf1}
f(t,x,v)= f_0(X_{0,t}(x,v) , V_{0,t}(x,v)) 
\end{equation}
whose density is computed by 
\begin{equation}\label{nonlinear-rho} 
\begin{aligned}
\rho(t,x) 
&= 
\int_{\RR^3} f_0(X_{0,t}(x,v) , V_{0,t}(x,v)) \, dv.
\end{aligned}
\end{equation}


\subsection{Nonlinear electric field}\label{sec-nonlinearE}


We now invert the Klein-Gordon equation \eqref{VKG2}. We first recall the following classical result. 

\begin{lemma} Let $\phi$ solve the Klein-Gordon equation 
$$(\Box +1)\phi = - \rho$$ 
with initial data $\phi_{\vert_{t=0}} = \phi_0$ and $\partial_t\phi_{\vert_{t=0}} = \phi_1$. Then, there holds the following representation 
\begin{equation}\label{rep-phi}
\phi(t,x) = G(t,x) \star_x\phi_1(x) + \partial_tG(t,x) \star_x\phi_0(x) -  G(t,x) \star_{t,x} \rho(t,x),
\end{equation}
where $G(t,x)$ denotes the Green function of the Klein-Gordon operator $\Box+1$ whose Fourier transform reads
\begin{equation}\label{def-Gr}
\FG(t,k) = \frac{1}{2i\langle k\rangle} \Big( e^{it \langle k\rangle} - e^{-it\langle k\rangle} \Big) = \frac{\sin (t\langle k\rangle)}{\langle k\rangle}
.\end{equation}
\end{lemma}
\begin{proof} Indeed, the solution to the Klein-Gordon equation in Fourier-Laplace variables is of the form 
$$
\Tphi(\lambda,k) = \frac{ \Fphi_1 + \lambda\Fphi_0 + \Trho(\lambda,k)}{\lambda^2 + 1 + k^2}.
$$
Thus, let 
$$ 
\FG(t,k) = \frac{1}{2\pi i} \int_{\Re \lambda \gg1} e^{\lambda t} \frac{1}{\lambda^2 + 1 + k^2} \; d\lambda
$$
be the Fourier transform of $G(t,x)$ in the variable $x$. 
Observe that the Green kernel $\frac{1}{\lambda^2 + 1 + k^2}$ is holomorphic and has poles at $\lambda_\pm(k) = \pm i \langle k\rangle$ with notation $\langle k \rangle =\sqrt{1+|k|^2}$, yielding \eqref{def-Gr}. 
Observe that $\FG(0,k) =0$ and $\partial_t \FG(0,k) = 1$. Therefore, upon taking the inverse Laplace transform, we obtain 
\begin{equation}
\label{dephi} 
\Fphi(t,k) = \FG(t,k) \Fphi_1(k) + \partial_t\FG(t,k) \Fphi_0(k) - \FG(t,k) \star_t \Frho(t,k),
\end{equation}
which gives the lemma. 
\end{proof}

\begin{corollary}[Nonlinear electric field]\label{cor-E} Introduce oscillatory kernels 
\begin{equation}\label{def-FG}\FG_\pm^{osc}(t,k) = e^{\lambda_\pm(k)t}a_\pm(k)\end{equation}
with $\lambda_\pm(k) = \pm i \langle k\rangle $ and $a_\pm(k) = \mp\frac{ i }{2} \langle k\rangle^{-1}.$
Then, the nonlinear electric field $E(t,x)$ of \eqref{VKG1}-\eqref{VKG2}can be expressed as  
\begin{equation}\label{rep-E2}
E(t,x) =  \sum_\pm E^{osc}_\pm(t,x), 
\end{equation}
where 
\begin{equation}\label{rep-E3}
 E^{osc}_\pm(t,x) =G_\pm^{osc}(t,x) \star_{x}  \nabla_x S_\pm^0(x) + G_\pm^{osc}(t,x) \star_{t,x} \nabla_x \rho(t,x),
 \end{equation}
for the nonlinear density $\rho(t,x)$ defined as in \eqref{nonlinear-rho}, and for initial data term $S_\pm^0(x)$ given by 
$$S_\pm^0(x) = - \phi_1(x) -\lambda_\pm(i\partial_x) \phi_0(x),$$
for $\phi_0, \phi_1$ being the initial data for the electric potential. 
\end{corollary}

\begin{proof} By definition, note that $\FG(t,k) = \sum_\pm \FG_\pm^{osc}(t,k)$, and so 
$$ \FG(t,k) \Fphi_1(k) + \partial_t\FG(t,k) \Fphi_0(k) = \sum_\pm \FG_\pm^{osc}(t,k) \Big[ \Fphi_1(k) + \lambda_\pm(k) \Fphi_0(k)  \Big].$$
The representation \eqref{rep-E2}-\eqref{rep-E3} thus follow from that of \eqref{rep-phi}, upon recalling that $E = -\nabla_x \phi$. 
\end{proof}

%
%
%

\subsection{Nonlinear iterative scheme}\label{sec-bootstrap}

In this section, we introduce a nonlinear iterative scheme to establish the large time behavior of the nonlinear solutions to the Vlasov-Klein-Gordon system \eqref{VKG1}-\eqref{VKG2}. The crucial idea is to bootstrap not only the decay but also the oscillation structure of the electric field. Namely, in view of \eqref{rep-E2}-\eqref{rep-E3}, we shall bootstrap the electric field of the oscillatory form 
\begin{equation}\label{E-bootstrap}E(t,x) =  \sum_\pm E^{osc}_\pm(t,x) + E^r(t,x),\end{equation}
\begin{equation}\label{Eosc-bootstrap} E^{osc}_\pm(t,x) =G_\pm^{osc}(t,x) \star_{x} F_0(x) + G_\pm^{osc}(t,x) \star_{t,x} F(t,x),
\end{equation}
in which $F_0(x)$ involves only initial data. Note that $F_0(x), F(t,x),$ and $E^r(t,x)$ are a gradient of some density, namely
\begin{equation}\label{grad-F} F_0(x) = \nabla_x S_0(x), \qquad F(t,x) = \nabla_x S(t,x).
\end{equation}

Fix $\alpha_0 \ge 8$ and a sufficiently small $\epsilon>0$. Our bootstrap assumptions are as follows:
\begin{itemize}

\item Decay assumptions: 
\begin{equation}\label{bootstrap-decayS}
\begin{aligned}
\|S(t)\|_{L^p_x}  + \|\partial_xS(t)\|_{L^p_x}   &\le \epsilon \langle t\rangle^{-3(1-1/p)} , \qquad p\in [1,\infty],
\\
\|E^r(t)\|_{L^p_x}  &\le \epsilon \langle t\rangle^{-3(1-1/p)} , \qquad p\in [1,\infty],
\\
\| E^{osc}_\pm(t)\|_{L^p_x} 
&\le \epsilon \langle t\rangle^{-3(1/2-1/p)} , \qquad p\in [2,\infty] .
\end{aligned}\end{equation}

\item Decay assumptions for higher derivatives: 
\begin{equation}\label{bootstrap-decaydaS}
\begin{aligned}
\|\partial_x^\alpha S(t)\|_{L^p_x} + \|\partial_x^\alpha E^r(t)\|_{L^p_x}   &\le \epsilon \langle t\rangle^{-3(1-1/p) + \delta_\alpha} , \qquad p\in [1,\infty], 
 \\
 \|\partial_x^{\alpha+1} E^r(t)\|_{L^2_x}  & \le \epsilon \langle t\rangle^{-3/2 + \delta_\alpha}
 \end{aligned}\end{equation}
for all $1\le |\alpha|\le |\alpha_0|-1$, with $\delta_\alpha = \frac{|\alpha|}{|\alpha_0|-1}$. 

\item Boundedness assumptions:
\begin{equation}\label{bootstrap-Hs}
\begin{aligned}
\| E^{osc}_\pm(t)\|_{H^{\alpha_0}_x} \le \epsilon , \qquad \| E^{osc}_\pm(t)\|_{H^{\alpha_0+1}_x} +  \| E^r(t)\|_{H^{\alpha_0+1}_x} + \| S(t)\|_{H_x^{\alpha_0}}
\le \epsilon \langle t\rangle^{\delta_1},
\end{aligned}
\end{equation}
with $\delta_1 = \frac{1}{|\alpha_0|-1}$. 
 \end{itemize}
Observe that the oscillatory field $E^{osc}_\pm(t,x)$ disperses in space like a Klein-Gordon wave at a rate of order $t^{-3/2}$, while the source density $S(t,x)$ disperses in space at a faster rate of order $t^{-3}$, dictated by the free transport dynamics. We observe that the boundedness of $E^{osc}_\pm$ in $L^2$ follows from the decay of $S(t)$ in $L^2$. Indeed, since the Fourier symbol of $\nabla_xG^{osc}_\pm$ is uniformly bounded, namely $ik \langle k\rangle^{-1}e^{\pm i \langle k\rangle}$, we get  
$$\| G^{osc}_\pm\star_{t,x} \nabla_x S\|_{L^2_x} \le \int_0^t \| G^{osc}_\pm(t-s)\star_{x} \nabla_x S(s)\|_{L^2_x} \; ds\lesssim \int_0^t \|S(s)\|_{L^2_x} \; ds \lesssim \epsilon.$$ 
In addition, by interpolation \eqref{dx-interpolate}, for $|\alpha|\le |\alpha_0|-1$, we observe that 
\begin{equation}\label{bootstrap-decaydxE}\| \partial^\alpha_x E^{osc}_\pm(t)\|_{L^\infty}  \lesssim\| E^{osc}_\pm(t)\|_{L^\infty}^{1-\frac{|\alpha|}{|\alpha_0|-1}} \| E^{osc}_\pm(t)\|_{H^{\alpha_0+1}}^{\frac{|\alpha|}{|\alpha_0|-1}} \lesssim \epsilon \langle t\rangle^{-\frac32 (1-\delta_{\alpha})+ \epsilon_\alpha},\end{equation}
where $\delta_{\alpha} = \frac{|\alpha|}{|\alpha_0|-1}$ and $\epsilon_\alpha = \frac{|\alpha|}{(|\alpha_0|-1)^2}$. However, we note that the decay assumptions for higher derivatives \eqref{bootstrap-decaydaS} of $S(t,x)$ do not follow from some simple interpolation inequalities between decay for low norms and boundedness for high norms. The structure of $S(t,x)$ plays an important role in deriving such decay estimates, see Section \ref{sec-sourceest}. Furthermore, the decay assumptions on $S(t)$ are not sufficient to derive the dispersive decay for $E^{osc}_\pm(t,x)$ in the $L^p_x$ norm 
through the spacetime convolution. In the nonlinear analysis, we need to further examine the structure of $S$ and the interplay between transport and oscillations to derive such a sharp dispersive decay for the field $E^{osc}_\pm(t,x)$, see Section \ref{sec-decayosc}. 

Finally, to better keep track of the oscillations of the electric field from \eqref{E-bootstrap}, we write in Fourier space 
\begin{equation} \label{defiFB}
\FE^{osc}_\pm(t,k) = e^{\lambda_{\pm}(k) t} \FB_\pm(t,k)
\end{equation}
where by construction 
\begin{equation}\label{def-Bpm}
\FB_\pm(t,k) = a_\pm(k) \FF_0(k)+  \int_0^t e^{-\lambda_\pm(k)s} a_\pm(k) \FF(s,k)\; ds, 
\end{equation}
recalling that $\lambda_\pm(k) = \pm i \langle k\rangle$ and $a_\pm(k) = \frac{\pm 1}{2i \langle k\rangle}$ as defined in \eqref{def-FG}. 
In particular, we note that 
\begin{equation}\label{dtBF}
e^{\lambda_\pm(k)t} \partial_t \FB_\pm(t,k)  = a_\pm(k) \FF(t,k) 
\end{equation}
which plays a role in the nonlinear analysis, since $F(t,x)$ decays faster, namely at quadratic order, 
when compared with that of $E^{osc}_\pm(t,x)$.

\subsection{The bootstrap argument}

Applying the standard local-in-time existence theory,  the bootstrap estimates \eqref{bootstrap-decayS}-\eqref{bootstrap-Hs} hold for $t \in [0,T]$ for some small $T>0$. Suppose that there is a finite time $T_*$ so that the bootstrap assumptions hold for all $t \in [0,T_*)$. It suffices to prove that they remain valid for $t=T_*$, and hence the solution is global in time and satisfies the stated bounds. We shall prove that there are universal constants $C_0, C_1$ so that for all $t \in [0,T_*]$, the electric field $E$ can be expressed as in \eqref{E-bootstrap}-\eqref{grad-F}, and in addition, 
there hold 
\begin{equation}\label{bootstrap-gdecayS}
\begin{aligned}
\|S(t)\|_{H^{\alpha_0}_x}   &\le \Big(C_0 \epsilon_0 + C_1\epsilon^2\Big) \langle t\rangle^{\delta_1},
\\
\| S(t)\|_{W^{1,p}_x}   &\le \Big(C_0 \epsilon_0 + C_1\epsilon^2\Big) \langle t\rangle^{-3(1-1/p)} , \qquad p\in [1,\infty],
\\
\| \partial_x^\alpha S(t)\|_{L^p_x}   &\le \Big(C_0 \epsilon_0 + C_1\epsilon^2\Big) \langle t\rangle^{-3(1-1/p)+\delta_\alpha} , \qquad p\in [1,\infty],
\end{aligned}\end{equation}
for $|\alpha|\le |\alpha_0|-1$, and 
\begin{equation}\label{bootstrap-gdecayE} 
\begin{aligned}
\|E^{osc}_\pm(t)\|_{H^{\alpha_0}_x}   &\le \Big(C_0 \epsilon_0 + C_1\epsilon^2\Big) ,
\\
\|E^{osc}_\pm(t)\|_{H^{\alpha_0+1}_x}   &\le \Big(C_0 \epsilon_0 + C_1\epsilon^2\Big) \langle t\rangle^{\delta_1},\\
\| E^{osc}_\pm(t)\|_{L^p_x} 
&\le \Big(C_0 \epsilon_0 + C_1\epsilon^2\Big) \langle t\rangle^{-3(1/2-1/p)} , \qquad p\in [2,\infty]. 
\end{aligned}\end{equation}
Similar bounds are obtained for $E^r(t,x)$. The decay and boundedness of $S(t)$ will be proved in Section \ref{sec-sourceest}, while the decay and boundedness of the electric field $E(t)$ will be proved in Section \ref{sec-decayosc}. The main analysis of this paper lies in the decay and boundedness of the nonlinear characteristic $X_{s,t}, V_{s,t}$, which will be established in Section \ref{sec-Char}. Finally, provided \eqref{bootstrap-gdecayS}-\eqref{bootstrap-gdecayE}, we may choose $\epsilon_0 \ll \epsilon \ll1$ so that 
$$ C_0 \epsilon_0 + C_1 \epsilon^2 < \epsilon. $$
That is, the bootstrap assumptions \eqref{bootstrap-decayS}-\eqref{bootstrap-Hs} indeed hold for $t=T_*$. The main theorem thus follows. 


\section{Characteristics}\label{sec-Char}



\subsection{Introduction}


Throughout this section, the electric field is assumed to satisfy all the bootstrap assumptions listed in Section \ref{sec-bootstrap}. The main analysis is then to derive decay and boundedness of the characteristics. Specifically, let $X_{s,t}(x,v)$ and $V_{s,t}(x,v)$ be the nonlinear characteristics. As the electric field satisfies $\|E(t)\|_{L^\infty_x}  \lesssim \epsilon \langle t \rangle^{-3/2}$, it directly follows from \eqref{ode-char2} that 
\begin{equation}\label{quick-bdV}
 \| V_{s,t} - v\|_{L^\infty_{x,v}} \lesssim  \epsilon \langle s\rangle^{-1/2} ,
 \end{equation}
namely, $V_{s,t}(x,v)$ remains close to $v$. 
However, such a bound on the velocity is too weak to precisely locate the position of the moving particles $X_{s,t} (x,v)$.
As $E$ has on oscillatory component, the velocities and the trajectories of the particle oscillate. 
We need to extract these oscillations in order to get better bounds on the velocities and a good localisation of the electrons. Throughout this section, we assume that the electric field is of the form
\begin{equation}\label{bootstrapE1}
E(t,x) = \sum_\pm E^{osc}_\pm(t,x) + E^r(t,x)
\end{equation} 
with the oscillatory electric field $E^{osc}_\pm$ of the form \eqref{Eosc-bootstrap}, plus a remainder $E^r(t,x)$ satisfying the bootstrap assumptions. 
In what follows, $X_{s,t}(x,v)$ and $V_{s,t}(x,v)$ are the nonlinear characteristic solving \eqref{ode-char}. 

\subsection{Collective electric fields}

For $j\ge 1$, we set  
\begin{equation}\label{def-Eoscj}
E^{osc,j}_\pm(t,x,v) = [E^{osc}_{\pm} \star_{x} \phi_{\pm,j}](t,x,v)
\end{equation}
where 
\begin{equation}\label{def-phipmj} \phi_{\pm,j}(x,v) = \int_{\RR^3} e^{ik\cdot x} \frac{1}{(\lambda_\pm(k) + ik\cdot \hv)^{j}} \; dk.
\end{equation}
Then, we obtain the following lemma. 
\begin{lemma}\label{lem-PEosc}
For any compact subset $K \Subset \RR^3$ and $\delta>0$, there hold  
\begin{equation}\label{Lp-convphi}
\begin{aligned}
\| \sup_{v\in K}  \phi_{\pm,j}(x,v) \star_x \partial_x^\alpha f \|_{L^p_x} &\lesssim \| f\|_{L^p_x} ,\qquad j>|\alpha|, 
\\
\| \sup_{v\in K}  \phi_{\pm,j}(x,v) \star_x \partial^\alpha_xf \|_{L^p_x} &\lesssim \| f\|^{1/2}_{L^p_x} \| \partial_x f\|^{1/2}_{L^p_x} , \qquad j = |\alpha|, 
\end{aligned}\end{equation}
for any $\alpha\ge 0$, $j\ge 1$, and $1\le p\le \infty$. 
\end{lemma}
\begin{proof}
Observe that for each $v$, the symbol $\lambda_\pm(k) + ik\cdot \hv$ never vanishes, recalling $\lambda_\pm(k) = \pm i \langle k\rangle$. Therefore, we may write 
\begin{equation}\label{expand-phi1} \frac{1}{\lambda_\pm(k) + ik \cdot \hv} = \mp i \langle k\rangle^{-1} \sum_{n\ge 0} (\pm 1)^n\langle k\rangle^{-n} (k\cdot \hat v)^n = \sum_{n\ge 0}  a_{\pm,n}(k) :: \hv^{\otimes n}\end{equation}
for $a_{\pm,n}(k) =  \mp i \langle k\rangle^{-1} (\pm 1)^n\langle k\rangle^{-n} k^{\otimes n}$, which are smooth Fourier multipliers and satisfy $|a_{\pm,n}(k)|\le \langle k\rangle^{-1}$. In the above the notation $k^{\otimes n}::\hv^{\otimes n} = (k\cdot \hv)^n$ is used only for sake of presentation. The series is absolutely converging for bounded $v$, since $\lambda_\pm(k) + ik \cdot \hv$ is bounded away from zero. The lemma thus follows from the results obtained in Lemma \ref{sec-FM}. 
\end{proof}

\begin{corollary}\label{cor-Ejosc}
Under the bootstrap assumptions on $E$, with $|\alpha_0| \ge 4$, there hold 
\begin{equation}\label{decay-Eoscj}
\begin{aligned}
\| \sup_{v\in K} \partial_x^\alpha \partial_v^\beta E^{osc,j}_\pm\|_{L^\infty_x} &\lesssim \epsilon \langle t\rangle^{-\frac32}, \qquad j > |\alpha|, 
\\
\| \sup_{v\in K}  \partial^\alpha_x \partial_v^\beta E^{osc,j}_\pm\|_{L^\infty_x} &\lesssim \epsilon \langle t\rangle^{-\frac32 + \delta_1 }, \qquad j = |\alpha|, 
\end{aligned}
\end{equation}
where $\delta_1 = \frac{1}{|\alpha_0|-1}$.
\end{corollary}
\begin{proof}
The corollary follows from the definition \eqref{def-Eoscj},  \eqref{Lp-convphi}, and the bootstrap assumptions on the decay of $E$ from \eqref{bootstrap-decayS} and that of its derivatives from \eqref{bootstrap-decaydxE}. Note that in the case $j=|\alpha|$, we in fact have 
$$\| \sup_{v\in K}  \partial^\alpha_x \partial_v^\beta E^{osc,j}_\pm\|_{L^\infty_x} \lesssim \epsilon \langle t\rangle^{-\frac32 + \frac34\delta_1 + \frac12 \epsilon_1},$$
with $\delta_1 = \frac{1}{|\alpha_0|-1}$ and $\epsilon_1= \frac{1}{(|\alpha_0|-1)^2}$. The stated estimate follows, since 
$\frac34\delta_1 + \frac12 \epsilon_1\le \delta_1,$
provided that $|\alpha_0|\ge 3$. 
\end{proof}

\subsection{Velocity description}

We first study oscillations in the velocity. 

\begin{proposition}\label{prop-charV}
Let $X_{s,t}(x,v), V_{s,t}(x,v)$ be the nonlinear characteristic solving \eqref{ode-char}. Then, there holds
 \begin{equation}\label{decomp-Vst}
 \begin{aligned}
V_{s,t}(x,v) &=
v -V^{osc}_{t,t}(x,v) + V^{osc}_{s,t}(x,v)
+ V^{tr}_{s,t}(x,v)
 \end{aligned}
 \end{equation}
where 
$$
\begin{aligned}
  V^{osc}_{s,t}(x,v) &= \sum_\pm E^{osc,1}_{\pm}(s,X_{s,t}(x,v),V_{s,t}(x,v)) 
 \\
V^{tr}_{s,t}(x,v)& = -  \int_s^t  
Q^{tr}(\tau,X_{\tau,t}(x,v), V_{\tau,t}(x,v)) \; d\tau
 \end{aligned}$$
in which $Q^{tr}(t,x,v)$ is defined by 
\begin{equation}\label{def-Qtr}
\begin{aligned} Q^{tr} &=\sum_\pm [a_\pm(i\partial_x) \phi_{\pm,1}\star_x F]
-  \sum_\pm \nabla_v \hv  E\cdot  \nabla_x E^{osc,2}_\pm + E^r.
\end{aligned} 
\end{equation}
In addition, there hold 
\begin{equation}\label{bounds-Vst}
\begin{aligned} 
\|V^{osc}_{s,t}\|_{L^\infty_{x,v}} \lesssim \epsilon \langle s\rangle^{-3/2} ,
\qquad   \|V^{tr}_{s,t}\|_{L^\infty_{x,v}} \lesssim \epsilon \langle s\rangle^{-2} , 
 \end{aligned}
\end{equation}
and 
\begin{equation}\label{decay-Vst}
\|V_{s,t}-v+V^{osc}_{t,t}\|_{L^\infty_{x,v}} \lesssim \epsilon \langle s\rangle^{-3/2} .
\end{equation}
 \end{proposition}

\begin{proof}
Recall that $E = \sum_\pm E^{osc}_\pm$, and therefore, from \eqref{ode-char2}, 
$$
\begin{aligned}V_{s,t}(x,v) =v   - \sum_\pm \int_s^t E^{osc}_\pm(\tau, X_{\tau,t}(x,v)) \; d\tau - \int_s^t E^r(\tau, X_{\tau,t}(x,v)) \; d\tau .
\end{aligned}$$
It suffices to study the integral of $E^{osc}_\pm$ along the characteristic. From \eqref{defiFB}, we write in Fourier space, $
\FE^{osc}_\pm(t,k) = e^{\lambda_{\pm}(k) t} \FB_\pm(t,k) .
$ Therefore, integrating by parts in $\tau$ and recalling that $\partial_\tau X_{\tau,t} = \hV_{\tau,t}$, we compute 
\begin{equation}\label{int-Eosc}
\begin{aligned}
 \int_s^t E^{osc}_{\pm}(\tau, X_{\tau,t}) \; d\tau
& = \int_s^t \int e^{\lambda_\pm(k) \tau + ik\cdot X_{\tau,t}}  \FB(\tau,k) \; dkd\tau 
\\
& = \int  \int_s^t\frac{d}{d\tau} (e^{\lambda_\pm(k) \tau + ik\cdot X_{\tau,t}})  \frac{\FB(\tau,k)}{\lambda_\pm(k) + ik \cdot \hV_{\tau,t}} \; d\tau dk 
\\
& = 
\int \Big[\frac{e^{\lambda_\pm(k) t + ik\cdot x}  \FB(t,k)}{\lambda_\pm(k) + ik \cdot \hv} -  \frac{e^{\lambda_\pm(k) s + ik\cdot X_{s,t}}\FB(s,k)}{\lambda_\pm(k) + ik \cdot \hV_{s,t}} \Big] \; dk\\
& \quad - \int_s^t \int e^{\lambda_\pm(k) \tau + ik\cdot X_{\tau,t}}  \Big[\frac{ \partial_\tau \FB(\tau,k)}{\lambda_\pm(k) + ik \cdot \hV_{\tau,t}} - \frac{ \FB(\tau,k) ik \cdot \partial_\tau \hV_{\tau,t}}{(\lambda_\pm(k) + ik \cdot \hV_{\tau,t})^2}\Big] \; dkd\tau .
\end{aligned}\end{equation}
Defining $\phi_{\pm,j}(x,v)$ as in \eqref{def-phipmj} and recalling from \eqref{defiFB} and \eqref{dtBF} that 
\begin{equation}\label{recall-FB}
 e^{\lambda_{\pm}(k) t} \FB_\pm(t,k) = \FE^{osc}_\pm(t,k) 
, \qquad e^{\lambda_\pm(k)t} \partial_t \FB_\pm(t,k)  = a_\pm(k) \FF(t,k), \end{equation}
and $\partial_\tau \hV_{\tau,t} =(\nabla_v \hv E)(\tau, X_{\tau,t},V_{\tau,t}),$
 we obtain \eqref{decomp-Vst}. It remain to prove the stated estimates. Indeed, 
the estimates on $V^{osc}_{s,t}$ follow from \eqref{decay-Eoscj} and the bootstrap assumption on $E^{osc}_\pm$. On the other hand, recalling that $F = \nabla_x S$ and using Lemma \ref{lem-PEosc}, we bound 
\begin{equation}\label{bdQtr}
\begin{aligned} 
 \| Q^{tr}(t)\|_{L^\infty_{x,v}} 
 &\lesssim \sup_v \sum_\pm \| a_\pm(i\partial_x) \phi_{\pm,1} \star_x \nabla_x S\|_{L^\infty_x}
+  \sum_\pm \|E\|_{L^\infty_x} \|\nabla_x E^{osc,2}_\pm \|_{L^\infty_{x,v}} + \| E^r (t)\|_{L^\infty_x}
\\
&\lesssim \|S\|_{W^{1,\infty}_x}
+  \|E\|^2_{L^\infty_x}  + \| E^r (t)\|_{L^\infty_x}
\\
&\lesssim \epsilon \langle t\rangle^{-3}
,\end{aligned} 
\end{equation}
upon using the bootstrap assumptions on $E$ and $S$. The estimates on $V^{tr}_{s,t}$ follow directly from that of $Q^{tr}(t)$. 
\end{proof}


\subsection{Characteristic description}


In this section, we study oscillations in the characteristic $X_{s,t}(x,v)$. In view of \eqref{ode-char2}, 
we write 
\begin{equation}\label{straightX}
X_{s,t}(x,v) = x - (t-s) \hPsi_{s,t}(x,v)
\end{equation}
where $\hPsi_{s,t}(x,v)$ is the velocity average defined by 
\begin{equation}\label{def-hatPsi}
\hPsi_{s,t} (x,v) = \frac{1}{t-s}\int_s^t \hV_{\tau,t}(x,v) \; d\tau
\end{equation}
in which we recall $\hv  = v/\sqrt{1+|v|^2}$. Note that $|\hPsi_{s,t} (x,v)|<1 $ and so $\Psi_{s,t} = \hPsi_{s,t}/\sqrt{1-|\hPsi_{s,t}|^2}$ is indeed well-defined. 
We first establish the following proposition which gives a precise description of oscillations in $\hPsi_{s,t}(x,v)$. 

\begin{proposition}\label{prop-charPsi} 
Let $X_{s,t}(x,v), V_{s,t}(x,v)$ be the nonlinear characteristic solving \eqref{ode-char}, and let $\hPsi_{s,t}(x,v)$ be defined as in \eqref{def-hatPsi}. Then, 
there hold 
\begin{equation}\label{decomp-Psist}
 \begin{aligned}
\Psi_{s,t}(x,v) &= v -  V^{osc}_{t,t}(x,v)+\Psi^{osc}_{s,t}(x,v)  + \Psi^{tr}_{s,t}(x,v) + \Psi^{R}_{s,t}(x,v)
\\
\hPsi_{s,t}(x,v) &= \hv - \hV^{osc}_{t,t}(x,v) + \hPsi^{osc}_{s,t}(x,v) + \hPsi^{tr}_{s,t}(x,v) + \Psi^{Q}_{s,t}(x,v)
 \end{aligned}
 \end{equation}
where  
$$
\begin{aligned}
\Psi^{osc}_{s,t}(x,v) &= \frac{1}{t-s}\sum_\pm \Big(E^{osc,2}_{\pm}(t,x,v) - E^{osc,2}_{\pm} (s,X_{s,t}(x,v),V_{s,t}(x,v)) \Big)
 \\
\Psi^{tr}_{s,t}(x,v)& = - \frac{1}{t-s} \int_s^t  
[ (\tau-s)  
Q^{tr} + Q^{tr,1}](\tau,X_{\tau,t}(x,v), V_{\tau,t}(x,v)) \; d\tau 
 \end{aligned}$$
and 
$$
\begin{aligned}
\Psi^Q_{s,t}(x,v) &=  \frac{1}{t-s}\int_s^t V_{\tau,t}^Q(x,v)\; d\tau, \qquad 
\\
\Psi^{R}_{s,t}(x,v)  &= \nabla_v \hv \Psi^Q_{s,t}(x,v)  + \cO(\langle v\rangle^5|\hPsi_{s,t}- \hv|^2) 
\end{aligned}$$
with $|V^Q_{s,t} (x,v)|\lesssim \langle v\rangle^{-2}|V_{s,t}-v|^2$, in which $Q^{tr}(t,x,v)$ is defined as in \eqref{def-Qtr} and 
\begin{equation}\label{def-Qtr1}
\begin{aligned} Q^{tr,1} &=\sum_\pm [a_\pm(i\partial_x) F\star_x \phi_{\pm,2}]
- 2 \sum_\pm \nabla_v \hv  E\cdot   \nabla_x E^{osc,3}_\pm .
\end{aligned} 
\end{equation}
Here in \eqref{decomp-Psist}, we abuse the notation to define $\hV^{osc}_{t,t} (x,v)=  \nabla_v \hv V_{t,t}^{osc}(x,v)$  and $\hPsi^{a}_{s,t} (x,v)=  \nabla_v \hv \Psi_{s,t}^a(x,v)$ for $a \in \{osc,tr\}$. 
In particular, there hold  
\begin{equation}\label{bounds-Wst}
\begin{aligned} \|  \Psi^{osc}_{s,t}\|_{L^\infty_{x,v}} &\lesssim \epsilon \langle s\rangle^{-3/2} (t-s)^{-1}, 
\quad  \| \Psi^{tr}_{s,t}\|_{L^\infty_{x,v}} \lesssim \epsilon \langle s\rangle^{-1} \langle t\rangle^{-1} , 
 \\
  \| \Psi^{Q}_{s,t}\|_{L^\infty_{x,v}} &\lesssim \epsilon^2 \langle s\rangle^{-2} \langle t\rangle^{-1} , 
 \qquad  \| \langle v\rangle^{-3}\Psi^{R}_{s,t}\|_{L^\infty_{x,v}} \lesssim \epsilon^2  \langle s\rangle^{-2}\langle t\rangle^{-1},
 \end{aligned}
\end{equation}
and 
\begin{equation}\label{decay-Xst}
\begin{aligned}
\| \hPsi_{s,t}-\hv + \hV^{osc}_{t,t}\|_{L^\infty_{x,v}} & \lesssim  \epsilon \langle s\rangle^{-1}(t-s)^{-1}.
\\
\| X_{s,t}-x + (t-s)(\hv - \hV^{osc}_{t,t})\|_{L^\infty_{x,v}} &\lesssim \epsilon \langle s\rangle^{-1}.
\end{aligned}
\end{equation}
\end{proposition}

\begin{proof} We first prove the expansion for $\hPsi_{s,t}(x,v)$. Note that $\hv  = v/\langle v\rangle$ is a regular and bounded function of $v$. In particular, $\nabla_v \hv  = \frac{1}{\langle v\rangle} (\mathbb{I} - \hv  \otimes \hv )$ and $|\nabla^2_v \hv|\lesssim \langle v\rangle^{-2}$. 
Hence, in view of \eqref{decomp-Vst}, we may write 
\begin{equation}\label{def-VQ}
\begin{aligned}
\hV_{s,t} (x,v)&=  \hv +\nabla_v \hv(V_{s,t} (x,v) - v) + V^Q_{\tau,t}(x,v)
\end{aligned}
\end{equation}
with a quadratic remainder $|V^Q_{\tau,t} (x,v)|\lesssim \langle v\rangle^{-2}|V_{s,t}-v|^2$. Therefore, by definition \eqref{def-hatPsi}, we compute 
$$
\begin{aligned}
\hPsi_{s,t}(x,v) &=  \hv 
+ \frac{1}{t-s}\int_s^t (\hV_{\tau,t}(x,v) - \hv ) \; d\tau
\\
&=\hv 
+ \frac{\nabla_v \hv }{t-s}\int_s^t (V_{\tau,t}(x,v) - v ) \; d\tau 
+ \frac{1}{t-s}\int_s^t V^Q_{\tau,t}(x,v) \; d\tau.
\end{aligned}
$$
Using Proposition \ref{prop-charV}, we compute the average of $V_{\tau,t}(x,v)$, 
$$
 \begin{aligned}
\frac{1}{t-s}\int_s^t V_{\tau,t}(x,v) &=
v -V^{osc}_{t,t}(x,v) + \frac{1}{t-s}\int_s^t [V^{osc}_{\tau,t}(x,v)
+ V^{tr}_{\tau,t}(x,v) ]\; d\tau
 \end{aligned}
$$
in which by construction, $  V^{osc}_{s,t} = \sum_\pm E^{osc,1}_{\pm} (s,X_{s,t},V_{s,t}) $. Therefore, similarly to what was done in the proof of Proposition \ref{prop-charV}, we have  
$$
\begin{aligned} \int_s^t V^{osc}_{\tau,t}(x,v) \;d\tau &= \int_s^t  \int \frac{e^{\lambda_\pm(k) \tau + ik\cdot X_{\tau,t}}\FB(\tau,k)}{\lambda_\pm(k) + ik \cdot \hV_{\tau,t}}\; dkd\tau
\\
& = \int \int_s^t \frac{d}{d\tau} (e^{\lambda_\pm(k) \tau + ik\cdot X_{\tau,t}})  \frac{\FB(\tau,k)}{(\lambda_\pm(k) + ik \cdot \hV_{\tau,t})^2} \; d\tau dk 
\\
& = 
\int \Big[\frac{e^{\lambda_\pm(k) t + ik\cdot x}  \FB(t,k)}{(\lambda_\pm(k) + ik \cdot \hv)^2} -  \frac{e^{\lambda_\pm(k) s + ik\cdot X_{s,t}}\FB(s,k)}{(\lambda_\pm(k) + ik \cdot \hV_{s,t})^2} \Big] \; dk\\
& \quad - \int_s^t \int e^{\lambda_\pm(k) \tau + ik\cdot X_{\tau,t}}  \Big[\frac{ \partial_\tau \FB(\tau,k)}{(\lambda_\pm(k) + ik \cdot \hV_{\tau,t})^2} - \frac{2 \FB(\tau,k) ik \cdot \partial_\tau \hV_{\tau,t}}{(\lambda_\pm(k) + ik \cdot \hV_{\tau,t})^3}\Big] \; dkd\tau .
\end{aligned}$$
The first two integrals equal to $E^{osc,2}_{\pm}(t,x,v) $ and $E^{osc,2}_{\pm} (s,X_{s,t}(x,v),V_{s,t}(x,v)) $, respectively, while the last integral is 
$$
\begin{aligned}  - \int_s^t  
Q^{tr,1}(\tau,X_{\tau,t}(x,v), V_{\tau,t}(x,v)) \; d\tau
\end{aligned}$$
having defined $Q^{tr,1}(t,x,v)$ as in \eqref{def-Qtr1}. In addition, by definition, we note 
$$
\begin{aligned}
\int_s^t V^{tr}_{\tau,t}\; d\tau
&= - \int_s^t\int_\tau^t  
Q^{tr}(\tau') \; d\tau' d\tau 
 = - \int_s^t(\tau-s)  
Q^{tr}(\tau) \; d\tau .
\end{aligned}$$
This gives the expansion for $\hPsi_{s,t}(x,v)$. The estimates on $Q^{tr}$ and $Q^{tr,1}$ follow from the bootstrap assumptions on $E$, which directly gives \eqref{bounds-Wst}.

Finally, we prove the expansion for $\Psi_{s,t}(x,v)$ from that of $\hPsi_{s,t}(x,v)$. Indeed, note that $1 - |\hv|^2 = \langle v\rangle^{-2}$, and observe that 
$$\hv \cdot \nabla_v \hv  = \frac{1}{\langle v\rangle} \hv \cdot (\mathbb{I} - \hv  \otimes \hv ) =  \frac{1}{\langle v\rangle} (1-|\hv|^2)\hv = \langle v\rangle^{-3} \hv.$$  
In addition, $V^Q_{s,t}(x,v)$ is of order $\cO(\epsilon^2\langle v\rangle^{-2})$. Therefore, 
$$
\begin{aligned}
 1 - |\hPsi_{s,t}(x,v)|^2 &= 1 - |\hv |^2 - 2 \hv \cdot (\hPsi_{s,t} - \hv) 
- |\hPsi_{s,t} - \hv|^2 \\ &= 1 - |\hv |^2 + \cO(\epsilon \langle v\rangle^{-2})
 \\ &= (1 - |\hv |^2) (1+ \cO(\epsilon)).
  \end{aligned}$$
In particular,  $1 - |\hPsi_{s,t}(x,v)|^2$ remains strictly positive for all $s,t,x,v$, and is close to $1 - |\hv |^2$. 
Therefore, $\Psi_{s,t}(x,v)$ is well-defined through the map $\hv \mapsto v = \hv / \sqrt{1-|\hv|^2}$, giving a similar expansion as that of $\hPsi_{s,t}(x,v)$. In particular, noting $\nabla_{\hv} v = \langle v\rangle (\mathbb{I} + v \otimes v)$ and $|\nabla^2_{\hv} v|\lesssim \langle v\rangle^5$, we have 
\begin{equation}\label{def-WR}
\Psi_{s,t} - v  = \nabla_{\hv} v (\hPsi_{s,t} - \hv) + \cO(\langle v\rangle^5|\hPsi_{s,t} - \hv|^2)
\end{equation}
in which the last term is put into the remainder $\Psi^R_{s,t}(x,v)$. This gives the expansion for $\Psi_{s,t}(x,v)$, upon noting that $\nabla_{\hv} v = (\nabla_v \hv)^{-1}$. 

Finally, in view of the definition, $Q^{tr,1}(t)$ satisfies the same estimates as done for $Q^{tr}(t)$ in \eqref{bdQtr}, yielding 
\begin{equation}\label{bdQtr1} 
\| Q^{tr}(t)\|_{L^\infty_{x,v}} + \| Q^{tr,1}(t)\|_{L^\infty_{x,v}} \lesssim \epsilon \langle t\rangle^{-3}.
\end{equation} 
The estimates in \eqref{bounds-Wst} thus follow at once from the definition and the above bounds. This yields the proposition.
\end{proof}


\subsection{Derivative bounds}\label{sec-dxchar}


In this section, we study derivatives of the characteristics. We prove the following proposition. 

\begin{proposition}\label{prop-Dchar} 
Let $X_{s,t}(x,v), V_{s,t}(x,v)$ be the nonlinear characteristic solving \eqref{ode-char}, and let $\hPsi_{s,t}(x,v)$ be defined as in \eqref{def-hatPsi}. Then,  for $0\le s\le t$, there hold 
\begin{equation}\label{Dxv-VWend}
 \| \partial_x \partial_v^\beta V^{osc}_{t,t} \|_{L^\infty_{x,v}}   \lesssim \epsilon \langle t\rangle^{-3/2+\delta_1} ,
\end{equation}
for any $\beta$, and
\begin{equation}\label{bounds-DVWst}
 \begin{aligned}
\| \partial_x (V_{s,t} - v + V^{osc}_{t,t} )\|_{L^\infty_{x,v}} &\lesssim \epsilon \langle s\rangle^{-3/2+\delta_1} ,
\\
 \| \partial_v (V_{s,t} - v + V^{osc}_{t,t} )\|_{L^\infty_{x,v}} &\lesssim \epsilon \langle s\rangle^{-3/2+\delta_1} (t-s) ,
 \\
\|(t-s)\partial_x (\hPsi_{s,t} - \hv + \hV^{osc}_{t,t})\|_{L^\infty_{x,v}} 
&\lesssim  \epsilon \langle s\rangle^{-1+\delta_1},
 \\
\| (t-s)\partial_v (\hPsi_{s,t} - \hv + \hV^{osc}_{t,t})\|_{L^\infty_{x,v}} 
&\lesssim  \epsilon \langle s\rangle^{-1+\delta_1} (t-s) ,
 \end{aligned}
 \end{equation}
 for $\delta_1 = \frac{1}{|\alpha_0|-1}$. 
In particular,
\begin{equation}\label{Jacobian-dxdv}
\begin{aligned}
|\det(\nabla_xX_{s,t}(x,v)) - 1| &\lesssim \epsilon , \qquad |\det(\nabla_vX_{s,t}(x,v))|& \gtrsim \langle v\rangle^{-5} |t-s|^3,
\end{aligned}\end{equation}
uniformly for all $x,v$. 

\end{proposition}


\begin{proof} By recalling that $V^{osc}_{t,t}(x,v) = E^{osc,1}_{\pm}(t,x,v)$, the estimate \eqref{Dxv-VWend} thus follows at once from the results in Corollary \ref{cor-Ejosc} and the bootstrap assumption on $E^{osc}_\pm(t,x)$. 
Next, in order to prove the bounds in \eqref{bounds-DVWst}, we shall first derive similar bounds on $\hPsi_{s,t}$. Indeed, recalling the expansion from Proposition \ref{prop-charPsi}, we compute
\begin{equation}\label{comp-dxPsi}
\begin{aligned}
\partial_x^\alpha \partial_v^\beta [ \hPsi_{s,t}(x,v) - \hv + \hV^{osc}_{t,t}(x,v)]  &= \partial_x^\alpha \partial_v^\beta \hPsi^{osc}_{s,t}(x,v)  +\partial_x^\alpha \partial_v^\beta \hPsi^{tr}_{s,t}(x,v)+\partial_x^\alpha \partial_v^\beta \Psi^Q_{s,t}(x,v) .
\end{aligned}
\end{equation}

\subsubsection*{Bounds on $x$-derivatives.}

We first prove the bounds on the spatial derivatives $\partial_x$. Note that the characteristics are nonlinear, and we shall need to close the estimates for derivatives. To this end, we introduce 
 \begin{equation}\label{def-zetadxV} 
 \begin{aligned}
 \zeta(s,t) &:= \sup_{0\le s\le \tau\le t} (t-\tau) \langle \tau\rangle^{1-\delta_1} \| \langle v\rangle\partial_x [\hPsi_{\tau,t}- \hv + \hV^{osc}_{t,t}] \|_{L^\infty_{x,v}}.
\end{aligned}\end{equation}
We first bound the derivatives of the characteristics $X_{s,t}, V_{s,t}$ in terms of $\zeta(s,t)$. Indeed, using \eqref{straightX}, \eqref{Dxv-VWend}, and the definition of $\zeta(s,t)$ in \eqref{def-zetadxV}, we note that
\begin{equation}\label{bd-dxXst}
\begin{aligned}
|\partial_x X_{s,t}(x,v)| &\le 1+ (t-s)|\partial_x \hPsi_{s,t}(x,v)| 
\lesssim 1+  \zeta(s,t).
  \end{aligned}\end{equation}
In order to bound $\partial_x V$, we use the expansion from Proposition \ref{prop-charV}, namely  
\begin{equation}\label{expVosc1}
\partial_x  V_{s,t}(x,v) = -\partial_x V^{osc}_{t,t}(x,v) + \partial_x V^{osc}_{s,t}(x,v)  +\partial_xV^{tr}_{s,t}(x,v).
\end{equation}
By definition, 
we bound 
$$
\begin{aligned}
|\partial_xV^{osc}_{s,t}(x,v)| &\le \sum_\pm |\partial_x X_{s,t} \partial_x E^{osc,1}_{\pm}(s)|+ \sum_\pm |\partial_x V_{s,t} \partial_v E^{osc,1}_{\pm}(s)| 
\\&\lesssim \epsilon \langle s\rangle^{-3/2+\delta_1} (1+  \zeta(s,t)) + \epsilon \langle s\rangle^{-3/2} \| \partial_x V_{s,t}\|_{L^\infty_{x,v}} 
.  \end{aligned}$$
 As for $\partial_xV^{tr}_{s,t}$, we first note by definition and Corollary \ref{cor-Ejosc} that   
\begin{equation}\label{bddxQtr}
\begin{aligned} 
\|\partial_x Q^{tr}(t)\|_{L^\infty_{x,v}}&\lesssim \|\partial_x[a_\pm(i\partial_x) F(t)\star_x \phi_{\pm,1}]\|_{L^\infty_{x,v}}
+ \|\partial_x(E\cdot  \nabla_x E^{osc,2}_\pm(t)) \|_{L^\infty_{x,v}} + \| \partial_x E^r(t)\|_{L^\infty_x}
\\
&\lesssim \|F(t)\|_{L^\infty_{x,v}} + \|E\|_{L^\infty_x} \|\partial_x E\|_{L^\infty_x}+ \| \partial_x E^r(t)\|_{L^\infty_x}.
\end{aligned} 
\end{equation}
Similar estimates hold for $\partial_v Q^{tr}(t)$, yielding 
\begin{equation}\label{Dxv-Qtr}
\begin{aligned} 
\|\partial_x Q^{tr}(t)\|_{L^\infty_{x,v}} &\lesssim \epsilon \langle t\rangle^{-3 + \delta_{1}}, \qquad \|\partial_v Q^{tr}(t)\|_{L^\infty_{x,v}} \lesssim \epsilon \langle t\rangle^{-3},
\end{aligned} 
\end{equation}
where $\delta_{1} = \frac{1}{|\alpha_0|-1}$. Therefore, we bound 
$$
\begin{aligned}
|\partial_xV^{tr}_{s,t}(x,v)| &\le  \int_s^t  \Big(
|\partial_x X_{\tau,t} \partial_xQ^{tr}(\tau)| + |\partial_xV_{\tau,t} \partial_v Q^{tr}(\tau)| \Big)\; d\tau .
\\
&\lesssim \epsilon\int_s^t  \Big( \langle \tau\rangle^{-3+ \delta_1} (1+  \zeta(\tau,t)) 
+ \langle \tau\rangle^{-3}\|\partial_xV_{\tau,t}\|_{L^\infty_{x,v}} \Big)\; d\tau 
\\
&\lesssim \epsilon\langle s\rangle^{-2+ \delta_{1}} (1+  \zeta(s,t)) 
+ \epsilon \langle s\rangle^{-2}\sup_{s\le \tau\le t}\|\partial_xV_{\tau,t}\|_{L^\infty_{x,v}} ,
  \end{aligned}$$
Putting these into \eqref{expVosc1}, we obtain 
$$
\begin{aligned}
|\partial_xV_{s,t}(x,v)| &\lesssim \epsilon \langle s\rangle^{-3/2+\delta_1} (1+  \zeta(s,t)) + \epsilon \langle s\rangle^{-3/2} \sup_{s\le \tau \le t}\| \partial_x V_{\tau,t}\|_{L^\infty_{x,v}} .
  \end{aligned}$$
Taking the supremum in $x,v,s$, and letting $\epsilon$ be sufficiently small, we arrive at 
\begin{equation}\label{est-dxVst}
\begin{aligned}
\|\partial_xV_{s,t}\|_{L^\infty_{x,v}} &\lesssim \epsilon \langle s\rangle^{-3/2+\delta_1} (1+  \zeta(s,t)).
  \end{aligned}
  \end{equation}

We are now ready to bound the derivatives of $\hPsi_{s,t}$. Precisely, using the expansions in \eqref{comp-dxPsi}, we shall prove that 
\begin{equation}\label{claim-dxVPsi}
|\langle v\rangle\partial_x \hPsi^{osc}_{s,t}(x,v)| +|\langle v\rangle\partial_x \hPsi^{tr}_{s,t}(x,v)|+|\langle v\rangle\partial_x \Psi^Q_{s,t}(x,v)|\lesssim \epsilon (t-s)^{-1} \langle s\rangle^{-1+\delta_1}(1+\zeta(s,t)),
\end{equation}
which would imply that 
$\zeta(s,t) \lesssim \epsilon,$ 
upon taking $\epsilon$ sufficiently small. By definition (see Proposition \ref{prop-charPsi}), together with \eqref{bd-dxXst} and \eqref{est-dxVst}, we bound
\begin{equation}\label{temp-Posc1}
\begin{aligned}
|\partial_x \hPsi^{osc}_{s,t}(x,v)| &\le \Big| \frac{\nabla_v \hv}{t-s}\sum_\pm \Big(\partial_x E^{osc,2}_{\pm}(t) - \partial_x X_{s,t}\cdot \nabla_x E^{osc,2}_{\pm} (s) -  \partial_x V_{s,t}\cdot \nabla_v E^{osc,2}_{\pm}(s) \Big)\Big|
\\&\lesssim \epsilon \langle v\rangle^{-1} (t-s)^{-1}\langle s\rangle^{-3/2} ( 1 + \zeta(s,t) ),
\end{aligned}\end{equation}
which proves the claim \eqref{claim-dxVPsi} for this term. Next, using again the estimates \eqref{bd-dxXst} and \eqref{est-dxVst}, we bound  
$$
\begin{aligned}
(t-s)|\langle v\rangle\partial_x \hPsi^{tr}_{s,t}(x,v)|& \le \int_s^t  
\Big( (\tau-s) (
|\partial_x Q^{tr}| 
+|\partial_v Q^{tr}|) + |\partial_xQ^{tr,1}|  + |\partial_vQ^{tr,1}| \Big) ( 1 + \zeta(\tau,t))\; d\tau.
\end{aligned}
$$
In view of \eqref{def-Qtr} and \eqref{def-Qtr1}, the quadratic terms $Q^{tr}(t)$ and $Q^{tr,1}(t)$ satisfy the same bounds as in \eqref{Dxv-Qtr}. 
Hence, we obtain 
\begin{equation}\label{temp-Posc2}
\begin{aligned}
|(t-s)\langle v\rangle\partial_x \hPsi^{tr}_{s,t}(x,v)|
&\lesssim \int_s^t \langle \tau\rangle^{-2+\delta_1}( 1 + \zeta(\tau,t))\; d\tau 
\\& \lesssim \epsilon \langle s\rangle^{-1+\delta_1} ( 1 + \zeta(s,t)),
\end{aligned}\end{equation}
proving the desired bound \eqref{claim-dxVPsi} for this term. 
On the other hand, recalling \eqref{def-VQ} and using \eqref{bounds-Vst} and \eqref{est-dxVst}, we bound 
$$ \begin{aligned} \langle v\rangle|\partial_x V^Q_{\tau,t}(x,v)| & \lesssim |V_{\tau,t}(x,v)-v| |\partial_x V_{\tau,t}(x,v)| + |V_{\tau,t}(x,v)-v|^2
\\& 
\lesssim \epsilon^2 \langle \tau\rangle^{-3 + \delta_1} (1 + \zeta(\tau,t)) .
\end{aligned}$$
Therefore, 
\begin{equation}\label{temp-Posc3}
\begin{aligned}
|(t-s)\langle v\rangle\partial_x \Psi^Q_{s,t}(x,v)|& \le \int_s^t \langle v\rangle|\partial_x V^Q_{\tau,t}(x,v)| \; d\tau 
\\& 
\lesssim \epsilon^2 \langle s\rangle^{-2+\delta_1} (1 + \zeta(s,t))
\end{aligned}
\end{equation}
giving the desired bound \eqref{claim-dxVPsi} on $\partial_x \Psi^Q_{s,t}(x,v)$.  This proves the claim \eqref{claim-dxVPsi} and therefore $\zeta(s,t) \lesssim \epsilon$ for all $0\le s\le t$. 
{ In particular, putting \eqref{temp-Posc1}, \eqref{temp-Posc2}, and \eqref{temp-Posc3} into the expansion \eqref{comp-dxPsi}, we obtain 
\begin{equation}\label{est-dxVosc}
\begin{aligned}
\| \langle v\rangle \partial^\alpha_x( \hPsi_{s,t}(x,v) - \hv + \hV^{osc}_{t,t}(x,v))\|_{L^\infty_{x,v}} 
\lesssim  \epsilon \langle s\rangle^{-3/2} (t-s)^{-1}+ \epsilon \langle s\rangle^{-1+\delta_1} (t-s)^{-1},
\end{aligned}
\end{equation}
for $|\alpha|\le 1$. Note that the bounds for $\alpha=0$ follow from the results obtained in Proposition \ref{prop-charPsi}. 

We now prove the first two estimates in \eqref{bounds-DVWst}. Using \eqref{def-WR}, we have 
$$
\partial_x (\Psi_{s,t} - v  +  V^{osc}_{t,t}) = \nabla_{\hv} v \partial_x (\hPsi_{s,t} - \hv +  \hV^{osc}_{t,t}) + \cO(\langle v\rangle^5|\hPsi_{s,t} - \hv| |\partial_x (\hPsi_{s,t} - \hv)|).
$$
The desired estimates on the first term follow from \eqref{est-dxVosc} and the fact that $|\nabla_{\hv} v| \le \langle v\rangle^3$. On the other hand, using again \eqref{est-dxVosc} and \eqref{Dxv-VWend}, we bound 
$$\| \langle v\rangle^2|\hPsi_{s,t} - \hv| |\partial_x (\hPsi_{s,t} - \hv)|\|_{L^\infty_{x,t}} \lesssim \Big[ \epsilon \langle s\rangle^{-3/2} (t-s)^{-1}+ \epsilon \langle s\rangle^{-1+\delta_1} \langle t\rangle^{-1}+ \epsilon \langle t\rangle^{-3/2}\Big]^2.$$
This proves the second estimate stated in \eqref{bounds-DVWst}. 

}

\subsubsection*{Bounds on $v$-derivatives.}

Next, we give bounds on the $v$-derivatives. Similarly as done above, we first introduce 
 \begin{equation}\label{def-zetadvV} 
 \begin{aligned}
 \zeta_1(s,t) &:= \sup_{0\le s\le \tau\le t} \langle \tau\rangle \| \langle v\rangle\partial_v [\hPsi_{\tau,t}- \hv + \hV^{osc}_{t,t}] \|_{L^\infty_{x,v}},
\end{aligned}\end{equation}
and proceed to bound the last two terms on the right of  \eqref{comp-dxPsi}. Using \eqref{straightX}, \eqref{Dxv-VWend}, and the definition of $\zeta_1(s,t)$, we note that
\begin{equation}\label{bd-dvXst}
\begin{aligned}
 |\partial_vX_{s,t}(x,v)| &\le (t-s)|\partial_v \hPsi_{s,t}(x,v)| 
 \lesssim (t-s)(1 + \zeta_1(s,t)).
\\
| \partial_v V_{s,t} (x,v)| &\lesssim 1 + \epsilon \langle s\rangle^{-1/2}  (t-s)(1 + \zeta_1(s,t)). 
  \end{aligned}\end{equation}
Therefore, by definition, together with \eqref{bd-dvXst}, we bound
$$
\begin{aligned}
|\partial_v \hPsi^{osc}_{s,t}(x,v)| &\le 
\frac{|\partial_v\nabla_v \hv|}{t-s}\sum_\pm \Big| E^{osc,2}_{\pm}(t,x,v) - E^{osc,2}_{\pm} (s,X_{s,t}(x,v),V_{s,t}(x,v)) \Big|
\\&\quad +\frac{|\nabla_v \hv|}{t-s}\sum_\pm \Big| \partial_v E^{osc,2}_{\pm}(t) - \partial_v X_{s,t}\cdot \nabla_x E^{osc,2}_{\pm} (s) -  \partial_v V_{s,t}\cdot \nabla_v E^{osc,2}_{\pm}(s) \Big|
\\&\lesssim  \epsilon \langle v\rangle^{-1}\langle s\rangle^{-3/2} ( 1 + \zeta_1(s,t)),
\end{aligned}$$
in which we note that there is no singularity at $s=t$, at which the right hand side vanishes.  This proves the desired bounds for this term. Finally, using \eqref{bd-dvXst} and the fact that $|t-\tau|\le |t-s|$ for $0\le s\le \tau \le t$, we compute 
$$
\begin{aligned}
|\langle v\rangle\partial_v \hPsi^{tr}_{s,t}(x,v)|& \le \int_s^t  
\Big( (\tau-s) (
|\partial_x Q^{tr}| 
+|\partial_v Q^{tr}|) + |\partial_xQ^{tr,1}|  + |\partial_vQ^{tr,1}| \Big) ( 1 + \zeta_1(\tau,t))\; d\tau 
\\&\quad + \frac{\langle v\rangle|\partial_v\nabla_v\hv |}{t-s} \int_s^t  
[ (\tau-s)  
|Q^{tr}| + |Q^{tr,1}|] \; d\tau  .
\end{aligned}
$$
The first integral is bounded by $C_0\epsilon \langle s\rangle^{-1} ( 1 + \zeta_1(s,t))$ as done above, while the second integral is bounded by $C_0\epsilon \langle s\rangle^{-2}$. 
On the other hand, recalling \eqref{def-VQ} and using \eqref{bd-dvXst}, we bound 
$$
\begin{aligned} 
\langle v\rangle|\partial_v V^Q_{\tau,t}(x,v)| & \lesssim |V_{\tau,t}(x,v)-v| |\partial_v V_{\tau,t}(x,v)| + |V_{\tau,t}(x,v)-v|^2
\\
&  \lesssim  \epsilon \langle \tau\rangle^{-3/2}  + \epsilon^2 \langle \tau\rangle^{-2} (t-\tau) (1 + \zeta_1(\tau,t)) ,
\end{aligned}
$$
which gives 
$$ |\langle v\rangle\partial_v \Psi^Q_{s,t}(x,v)| \lesssim  \frac{1}{t-s} \int_s^t \langle v\rangle|\partial_v V^Q_{\tau,t}(x,v)| \; d\tau \lesssim  \epsilon \langle s\rangle^{-1} (1 + \zeta_1(s,t)) .$$
{
This completes the proof of the bounds on $\partial_v \hPsi_{s,t}(x,v)$, namely  
\begin{equation}\label{est-dvVosc}
\| \langle v\rangle\partial_v (\hPsi_{s,t} - \hv + \hV^{osc}_{t,t})\|_{L^\infty_{x,v}} 
\lesssim  \epsilon \langle s\rangle^{-1} ,\end{equation}
upon taking $\epsilon$ sufficiently small. To prove the last estimate in \eqref{bounds-DVWst}, we use \eqref{def-WR} and write
$$
\begin{aligned}
\partial_v (\Psi_{s,t} - v  +  V^{osc}_{t,t}) &=  (\partial_v\nabla_{\hv} v) (\hPsi_{s,t} - \hv +  \hV^{osc}_{t,t})  + \nabla_{\hv} v \partial_v (\hPsi_{s,t} - \hv +  \hV^{osc}_{t,t}) 
\\&\quad  + \cO(\langle v\rangle^4|\hPsi_{s,t} - \hv|^2) + \cO(\langle v\rangle^5|\hPsi_{s,t} - \hv| |\partial_v (\hPsi_{s,t} - \hv)|).
\end{aligned}
$$
Using \eqref{est-dxVosc} and \eqref{est-dvVosc}, we thus obtain the last estimate stated in \eqref{bounds-DVWst}. Finally, as for the third estimate on $\partial_v V_{s,t}$ stated in \eqref{bounds-DVWst}, we simply take the derivative of \eqref{decomp-Vst} and using \eqref{bd-dvXst}, with $\zeta_1(s,t) \lesssim \epsilon$. The bounds \eqref{bounds-DVWst} thus follow. 
}

Finally, the estimates \eqref{Jacobian-dxdv} follow directly from \eqref{straightX} and the bounds \eqref{Dxv-VWend}-\eqref{bounds-DVWst}. This completes the proof of the proposition. 
\end{proof}


\subsection{Decay for higher derivatives}\label{sec-Hdxchar}


In this section, we establish the decay and boundedness of higher derivatives of the characteristics. We prove the following proposition. 

\begin{proposition}\label{prop-HDchar} 
Fix $|\alpha_0| \ge 4$. Let $X_{s,t}(x,v), V_{s,t}(x,v)$ be the nonlinear characteristic solving \eqref{ode-char}. Then, for $0\le s\le t$ and for $1\le |\alpha|\le |\alpha_0|-1$, there hold 
\begin{equation}\label{decayXV}
 \begin{aligned}
\| \partial^\alpha_x(V_{s,t}-v)\|_{L^\infty_{x,v}}  
&\lesssim \epsilon \langle s\rangle^{-\frac32 (1- \delta_{\alpha})} \langle t\rangle^{\delta_\alpha}, 
\\
\| \partial^\alpha_x(X_{s,t}-x)\|_{L^\infty_{x,v}}  
&\lesssim \epsilon \langle s\rangle^{-\frac12 (1- \delta_{\alpha})} \langle t\rangle^{\delta_\alpha},
 \end{aligned}
 \end{equation}
with $\delta_\alpha = \frac{|\alpha|}{|\alpha_0|-1}$. 
\end{proposition}

We shall prove Proposition \ref{prop-HDchar} by continuous induction. Namely, we introduce the iterative bootstrap function
\begin{equation}\label{lownorm}
 \begin{aligned}
\zeta(s,t) &= \sup_{s\le \tau \le t} \sum_{|\alpha| < |\alpha_0|}\Big[\langle s\rangle^{\frac32 (1- \delta_{\alpha})} \langle t\rangle^{-\delta_\alpha}\| \partial^\alpha_x (V_{s,t}-v+V^{osc}_{t,t})\|_{L^\infty_{x,v}} 
\\&\qquad + \langle s\rangle^{1- \delta_{\alpha}} (t-s) \langle t\rangle^{-\delta_\alpha} \| \partial^\alpha_x(\hPsi_{s,t}-\hv+V^{osc}_{t,t})\|_{L^\infty_{x,v}} \Big]
 \end{aligned}
\end{equation}
with $\delta_\alpha = \frac{|\alpha|}{|\alpha_0|-1}$. Specifically, we shall prove that 
\begin{equation}\label{claimXV1}
\zeta(s,t) \le C_0\epsilon + C_0\epsilon \zeta(s,t)
\end{equation}
for some universal constant $C_0$, which would then yield $\zeta(s,t) \le 2C_0 \epsilon$, upon taking $\epsilon$ sufficiently small. The boundedness and decay of the characteristics then follow from that of $\zeta(s,t)$. 

\subsubsection*{Consequences of the bootstrap function.}
We first recall that $V^{osc}_{t,t}(x,v)= E^{osc,1}_{\pm}(t,x,v)$, and therefore by the bootstrap assumption on $E$ from Section \ref{sec-bootstrap}, 
\begin{equation}\label{bdVosct} \sup_v\| V^{osc}_{t,t}\|_{H^{\alpha_0+1}_{x}} \lesssim \| E^{osc}_\pm(t)\|_{H^{\alpha_0}} \lesssim \epsilon.\end{equation}
This, together with the decay estimate \eqref{decay-Eoscj} and the interpolation inequality \eqref{dx-interpolate}, yields  
\begin{equation}\label{daVosct}
\| \partial^\alpha_x V^{osc}_{t,t}\|_{L^\infty_{x,v}} + \| \partial^\alpha_x E^{osc,1}_{\pm}(t)\|_{L^\infty_{x,v}}  \lesssim \epsilon \langle t\rangle^{-\frac32 (1- \delta_{\alpha})}, 
\end{equation}
which hold for all $|\alpha|\le |\alpha_0|-1$, where $\delta_{\alpha} = \frac{|\alpha|}{|\alpha_0|-1}$. Next, in view of \eqref{daVosct} and the bootstrap funciton \eqref{lownorm}, as long as $\zeta(s,t)$ remains finite, we have
$$
 \begin{aligned}
\| \partial^\alpha_x ( V_{s,t}-v)\|_{L^\infty_{x,v}} &\lesssim \epsilon \langle t\rangle^{-\frac32 (1- \delta_{\alpha})} + \langle s\rangle^{-\frac32 (1- \delta_{\alpha})} \langle t\rangle^{\delta_\alpha}\zeta(s,t)
\\
 \| \partial^\alpha_x(\hPsi_{s,t}-\hv)\|_{L^\infty_{x,v}}   &\lesssim   \epsilon \langle t\rangle^{-\frac32 (1- \delta_{\alpha})} + \epsilon \langle s\rangle^{- (1- \delta_{\alpha})}(t-s)^{-1} \langle t\rangle^{\delta_\alpha}\zeta(s,t). \end{aligned}
$$
Therefore, for $0\le s\le t$, we have
\begin{equation}\label{inductV}
\begin{aligned}
\| \partial^\alpha_x(V_{s,t}-v)\|_{L^\infty_{x,v}}  
&\lesssim \langle s\rangle^{-\frac32 (1- \delta_{\alpha})} \langle t\rangle^{\delta_\alpha}(\epsilon+\zeta(s,t)),
\end{aligned}\end{equation}
for all $ |\alpha|\le |\alpha_0|-1$. Similarly, in view of \eqref{straightX}, \eqref{daVosct}, and \eqref{lownorm}, we have 
\begin{equation}\label{inductX}
\begin{aligned}
\| \partial^\alpha_x(X_{s,t}-x)\|_{L^\infty_{x,v}}  
&\le (t-s) \|\partial_x^\alpha \hPsi_{s,t}\|_{L^\infty_{x,v}}
\\& \lesssim   \epsilon \langle t\rangle^{-\frac12+ \frac32\delta_{\alpha}} + 
 \langle s\rangle^{-(1- \delta_{\alpha})}\langle t\rangle^{\delta_\alpha}\zeta(s,t)
\\& \lesssim   \epsilon \langle t\rangle^{-\frac12 (1- \delta_{\alpha})} \langle t\rangle^{\delta_\alpha}+ 
\langle s\rangle^{- (1- \delta_{\alpha})} \langle t\rangle^{\delta_\alpha}\zeta(s,t)
\\
&\lesssim \langle s\rangle^{-\frac12 (1- \delta_{\alpha})} \langle t\rangle^{\delta_\alpha}(\epsilon+\zeta(s,t)),
\end{aligned}\end{equation}
for all $1\le |\alpha|\le |\alpha_0|-1$. The estimates \eqref{inductV}-\eqref{inductX} yield the desired bounds \eqref{decayXV}, provided that $\zeta(s,t)$ remains finite. 

We now use the expansion from Proposition \ref{prop-charV} and Proposition \ref{prop-charPsi} to bound $\zeta(s,t)$, which we compute
\begin{equation}\label{expPVst}
\begin{aligned}
\partial^\alpha_x  \Big(V_{s,t}(x,v)- v + V^{osc}_{t,t}(x,v)\Big) &=  \partial_x^\alpha V^{osc}_{s,t}(x,v)  +\partial^\alpha_xV^{tr}_{s,t}(x,v),
\\
\partial_x^\alpha \Big(\hPsi_{s,t}(x,v) - \hv + \hV^{osc}_{t,t}(x,v)\Big )& = \partial_x^\alpha \hPsi^{osc}_{s,t}(x,v)  +\partial_x^\alpha \hPsi^{tr}_{s,t}(x,v)+\partial_x^\alpha \Psi^Q_{s,t}(x,v) .
\end{aligned}
\end{equation}
We shall estimate each term in \eqref{expPVst}. To proceed, we first recall 
the Fa\`a di Bruno's formula for composite functions (see, e.g., \cite{BCD}), namely 
\begin{equation}\label{Faa} \partial_x^n f(u) = \sum C_{k_j,n} \partial_x^k f (u) \prod_{j=1}^n (\partial_x^j u)^{k_j}  
\end{equation}
where $k_j \ge 0$, $\sum k_j = k$, and the summation is over all the partitions $\{ k_j\}_{j=1}^n$ of $n$ so that $\sum_j jk_j = n$. The extension to the multidimensional case is similar. 

\subsubsection*{Bounds on $\partial_x^\alpha V^{osc}_{s,t}(x,v)$.}
Recalling $ V^{osc}_{s,t} (x,v)= E^{osc,1}_{\pm}(s,X_{s,t}(x,v),V_{s,t}(x,v))$, we shall prove that 
\begin{equation}\label{est-daVosc}
\begin{aligned}  
\|\partial_x^\alpha V^{osc}_{s,t}\|_{L^\infty_{x,v}} 
&\lesssim 
\epsilon\langle s\rangle^{- \frac32(1-\delta_\alpha)}\langle t\rangle^{\delta_\alpha} (\epsilon+\zeta(s,t)).
\end{aligned}
\end{equation}
Indeed, using \eqref{Faa}, we compute 
\begin{equation}\label{daVosci}
\begin{aligned}  
\partial_x^\alpha V^{osc}_{s,t}= \sum_{1\le |\mu|\le |\alpha|} C_{\beta,\mu}  
\partial_x^\mu E^{osc,1}_{\pm} (s) \prod_{1\le |\beta|\le |\alpha|} (\partial_x^\beta X_{s,t})^{k_\beta} + \mathrm{l.o.t.}
\end{aligned}\end{equation}
where $\sum k_\beta = |\mu|$ and $\sum |\beta|k_\beta = |\alpha|$. Here in the above expression, $\mathrm{l.o.t.}$ denotes the lower order terms, namely all those terms that involve one or more spatial derivatives of $V_{s,t}$. These terms decay at the same or better rate. Therefore, recalling \eqref{daVosct}, we bound 
$$
\begin{aligned}  
\Big \|\sum_{1\le |\mu|\le |\alpha|} C_{\beta,\mu}  
\partial_x^\mu E^{osc,1}_{\pm} (s) \prod_{1\le |\beta|\le |\alpha|} (\partial_x^\beta X_{s,t})^{k_\beta}\Big \|_{L^\infty_{x,v}}
&\lesssim \epsilon \sum_{1\le |\mu|\le |\alpha|}\langle s\rangle^{-\frac32 (1- \delta_{\mu})}  \prod_{1\le |\beta|\le |\alpha|} \|\partial_x^\beta X_{s,t}\|_{L^\infty_{x,v}}^{k_\beta}.
  \end{aligned}$$
Note that $\partial^\beta_xX_{s,t}$ is uniformly bounded for $|\beta|=1$, see Proposition \ref{prop-Dchar}. 
Therefore, using \eqref{inductX}, we bound 
\begin{equation}\label{supprodX} 
\begin{aligned}  
\prod_{2\le |\beta|\le |\alpha|} \|\partial_x^\beta X_{s,t}\|_{L^\infty_{x,v}}^{k_\beta} 
&\lesssim  
 \prod_{2\le |\beta|\le |\alpha|} \langle s\rangle^{-\frac{k_\beta}2 (1-\delta_\beta)}\langle t\rangle^{\delta_\beta k_\beta} (\epsilon+\zeta(s,t))^{k_\beta}
\\
&\lesssim  
\langle s\rangle^{-\frac12\sum^*_\beta k_\beta + \frac12 \sum^*_\beta\delta_\beta k_\beta}\langle t\rangle^{\sum^*_\beta \delta_\beta k_\beta} (\epsilon+\zeta(s,t))^{\sum^*_\beta k_\beta}
\end{aligned}
\end{equation}
in which the summation $\sum_\beta^*$ is taken over $2\le |\beta|\le |\alpha|$ (noting $k_\beta \ge 0$). Now, using the fact that  $\sum_\beta k_\beta = |\mu|$ and $\sum |\beta|k_\beta = |\alpha|$, we note that  
 $  \sum_\beta \delta_{\beta}k_\beta = \delta_{\alpha}$. Therefore,  we have
$$1\le \sum^*_\beta k_\beta \le |\mu|, \qquad \sum^*_\beta \delta_\beta k_\beta \le \delta_\alpha,$$
and hence, 
\begin{equation}\label{supprodX0} 
\begin{aligned}
\prod_{2\le |\beta|\le |\alpha|} \|\partial_x^\beta X_{s,t}\|_{L^\infty_{x,v}}^{k_\beta} 
&\lesssim \langle s\rangle^{-\frac{1}2 (1-\delta_\alpha)}\langle t\rangle^{\delta_\alpha} (\epsilon+\zeta(s,t))
.\end{aligned}
\end{equation}
Putting these estimates into \eqref{daVosci}, we get 
$$
\begin{aligned}  
\|\partial_x^\alpha V^{osc}_{s,t}\|_{L^\infty_{x,v}} 
&\lesssim  
\epsilon \sum_{1\le |\mu|\le |\alpha|}\langle s\rangle^{-\frac32 (1- \delta_{\mu})} \Big[ 1 +  \langle s\rangle^{-\frac{1}2 (1-\delta_\alpha)}\langle t\rangle^{\delta_\alpha} (\epsilon+\zeta(s,t)) \Big]
\\
&\lesssim  \epsilon \langle s\rangle^{-\frac32 (1- \delta_{\alpha})} +  
\epsilon\langle s\rangle^{- 2(1-\delta_\alpha)}\langle t\rangle^{\delta_\alpha} (\epsilon+\zeta(s,t))
\\
&\lesssim 
\epsilon\langle s\rangle^{- \frac32(1-\delta_\alpha)}\langle t\rangle^{\delta_\alpha} (\epsilon+\zeta(s,t))
\end{aligned}
$$
in which we have used that $ \delta_\mu \le \delta_\alpha \le 1$ for $|\mu|\le |\alpha|\le |\alpha_0|-1$. This proves \eqref{est-daVosc}. 

\subsubsection*{Bounds on $\partial_x^\alpha V^{tr}_{s,t}(x,v)$.}

Next, we bound on $ \partial_x^\alpha V^{tr}_{s,t}(x,v)$, namely we shall prove 
\begin{equation}\label{est-daVtr}
\begin{aligned}
\| \partial_x^\alpha V^{tr}_{s,t} \|_{L^\infty_{x,v}} 
\lesssim \epsilon \langle s\rangle^{-2 (1-\delta_{\alpha})}\langle t\rangle^{\delta_\alpha} (\epsilon+\zeta(s,t)).
\end{aligned}
\end{equation}
By definition, we recall 
\begin{equation}\label{recalldaVtr}
\partial_x^\alpha V^{tr}_{s,t}(x,v) = -  \int_s^t  
\partial_x^\alpha[Q^{tr}(\tau,X_{\tau,t}(x,v), V_{\tau,t}(x,v))] \; d\tau.
\end{equation}
As before, since $\partial_x^\beta V_{s,t}$ (and $\partial^\gamma_v Q^{tr}$) satisfy better estimates than those for $\partial_x^\beta X_{s,t}$ (and $\partial^\gamma_x Q^{tr}$, respectively), we shall therefore focus terms that only involve derivatives of $X_{s,t}$. Namely, we bound 
\begin{equation}\label{compItr}
I^{tr,\alpha}_{s,t} = \sum_{1\le |\mu|\le |\alpha|} C_{\beta,\mu}  
\int_s^t \partial_x^\mu Q^{tr}(\tau) \prod_{1\le |\beta|\le |\alpha|} (\partial_x^\beta X_{\tau,t}(x,v))^{k_\beta} \; d\tau.
\end{equation}
By definition \eqref{def-Qtr}, we compute 
\begin{equation}\label{comdaQtr}
\begin{aligned} 
\partial_x^\mu Q^{tr}&=\sum_\pm \partial_x^\mu [a_\pm(i\partial_x) F\star_x \phi_{\pm,1}]
-  \sum_\pm \nabla_v \hv  \partial_x^\mu(E\cdot  \nabla_x E^{osc,2}_\pm) + \partial_x^\mu  E^r.
\end{aligned} 
\end{equation}
Observe that under the boundedness of $V_{s,t}, X_{s,t}$, see \eqref{inductV}-\eqref{inductX}, we have 
$$ \|  a_\pm(i\partial_x) \partial_x^\mu F(t)\star_x \phi_{\pm,1}\|_{L^\infty_{x,v}} \lesssim \| \partial_x^{\mu} S(t)\|_{L^\infty_{x,v}} \lesssim \epsilon \langle t\rangle^{-3 + \delta_\mu}.$$
On the other hand, note that $\|  \nabla_x E^{osc,2}_\pm\|_{L^\infty} \lesssim \| E^{osc}_\pm \|_{L^\infty}$ and $\|  \nabla_x E^{osc,2}_\pm\|_{H^{\alpha_0+1}} \lesssim \| E^{osc}_\pm \|_{H^{\alpha_0}}$. Therefore, 
using the interpolation \eqref{dx-interpolate}, with $1\le |\mu|\le |\alpha_0|-1$, and the inequality 
$$\| uv\|_{H^{\alpha_0}} \lesssim \| u\|_{L^\infty} \| v\|_{H^{\alpha_0}}+ \| v\|_{L^\infty} \| u\|_{H^{\alpha_0}},$$ 
we obtain 
$$
\begin{aligned} 
\| \partial_x^\mu(E\cdot  \nabla_x E^{osc,2}_\pm(t))\|_{L^\infty_{x,v}} 
& \lesssim 
\| E\cdot  \nabla_x E^{osc,2}_\pm(t)\|_{L^\infty_x}^{1-\frac{|\mu|}{|\alpha_0|-1}} \| E\cdot  \nabla_x E^{osc,2}_\pm(t)\|_{H^{\alpha_0+1}_x}^{\frac{|\mu|}{|\alpha_0|-1}}
\\
& \lesssim 
\| E\|_{L^\infty_x}^{2-\frac{|\mu|}{|\alpha_0|-1}} \| E\|_{H^{\alpha_0+1}_x}^{\frac{|\mu|}{|\alpha_0|-1}}
 \lesssim 
 \epsilon^2 \langle t\rangle^{-3 +\frac32\delta_{\mu}+ \epsilon_\mu} 
 \end{aligned} 
$$
with $\epsilon_\mu = \delta_1 \delta_\mu$ and $\delta_\mu = \frac{|\mu|}{|\alpha_0| -1}$, in which we have used the bootstrap assumption $\| E^{osc}_\pm(t)\|_{H^{\alpha_0+1}_x} 
\le \epsilon \langle t\rangle^{\delta_1}$. Therefore, together with \eqref{bootstrap-decaydaS}, we have 
\begin{equation}\label{dmuQtr}
\| \partial_x^\mu Q^{tr}(t)\|_{L^\infty_{x,v}}  \lesssim 
 \epsilon^2 \langle t\rangle^{-3 +\frac32\delta_{\mu}+ \epsilon_\mu} 
\end{equation}
and 
\begin{equation}\label{est-Itrst}
\begin{aligned}
\|I^{tr,\alpha}_{s,t} \|_{L^\infty_{x,v}} 
&\lesssim 
\sum_{1\le |\mu|\le |\alpha|}  
\int_s^t \| \partial_x^\mu Q^{tr}(\tau) \|_{L^\infty_{x,v}}\prod_{1\le |\beta|\le |\alpha|} \| \partial_x^\beta X_{\tau,t}\|_{L^\infty_{x,v}}^{k_\beta} \; d\tau
\\&\lesssim \epsilon\sum_{1\le |\mu|\le |\alpha|}  
\int_s^t  \langle \tau\rangle^{-3 +\frac32\delta_{\mu}+\epsilon_\mu} \prod_{1\le |\beta|\le |\alpha|} \| \partial_x^\beta X_{\tau,t}\|_{L^\infty_{x,v}}^{k_\beta} \; d\tau
.\end{aligned}
\end{equation}
We first treat the case when $k_\beta = 0$ for all $|\beta| \ge 2$ (namely, only involving the first derivatives of $X_{s,t}$). In this case, using $\| \partial_x^\beta X_{\tau,t}\|_{L^\infty_{x,v}}^{k_\beta}\lesssim 1$ for $|\beta|=1$, and noting that $\delta_\mu \le \delta_\alpha$, for each $|\mu|\le |\alpha|$, we bound  
$$
\int_s^t  \langle \tau\rangle^{-3 +\frac32\delta_{\mu}+\epsilon_\mu} \; d\tau
\lesssim \langle s\rangle^{-2 +\frac32\delta_\alpha+\epsilon_\alpha}
\lesssim \langle s\rangle^{-2(1-\delta_\alpha)},
$$
upon noting $\epsilon_\alpha = \delta_1\delta_\alpha \le \frac12 \delta_\alpha$. 
We next treat the case when $k_\beta \ge 1$ for some $|\beta|\ge 2$. Using \eqref{supprodX0}, for each $|\mu|\le |\alpha|$, we bound 
$$
\begin{aligned}
\int_s^t  \langle \tau\rangle^{-3 +\frac32\delta_{\mu}} \prod_{2\le |\beta|\le |\alpha|} \| \partial_x^\beta X_{\tau,t}\|_{L^\infty_{x,v}}^{k_\beta} \; d\tau
&\lesssim \int_s^t  \langle \tau\rangle^{-3 +\frac32\delta_{\mu}+\epsilon_\mu} \langle \tau \rangle^{-\frac{1}2 (1-\delta_\alpha)}\langle t\rangle^{\delta_\alpha} (\epsilon+\zeta(\tau,t)) \; d\tau
\\
&\lesssim \langle s\rangle^{-2 +2\delta_{\alpha}} \langle s\rangle^{-\frac{1}2 (1-\delta_\alpha)}\langle t\rangle^{\delta_\alpha} (\epsilon+\zeta(s,t)) 
\\
&\lesssim \langle s\rangle^{-\frac52(1-\delta_\alpha)}\langle t\rangle^{\delta_\alpha} (\epsilon+\zeta(s,t)) 
.\end{aligned}
$$
This completes the proof of \eqref{est-daVtr}. 

\subsubsection*{Bounds on $\partial_x^\alpha \hPsi^{osc}_{s,t}(x,v)$.}

Next, we prove bounds on $\partial_x^\alpha \hPsi^{osc}_{s,t}(x,v)$, which we recall
\begin{equation}\label{comPosc}
\begin{aligned}
(t-s)\partial_x^\alpha\hPsi^{osc}_{s,t}(x,v) &=  \nabla_v \hv\sum_\pm \partial_x^\alpha\Big(E^{osc,2}_{\pm}(t,x,v) - E^{osc,2}_{\pm} (s,X_{s,t}(x,v),V_{s,t}(x,v)) \Big).
 \end{aligned}\end{equation}
Since $E^{osc,2}(s)$ satisfies the same (or better) estimates than those for $ V^{osc}_{s,t} = E^{osc,1}_{\pm}(s)$, following the proof of \eqref{est-daVosc}, we thus obtain   
\begin{equation}\label{est-daPosc}
\begin{aligned}  
(t-s)\|\partial_x^\alpha\hPsi^{osc}_{s,t}\|_{L^\infty_{x,v}} 
 \lesssim 
 \epsilon\langle s\rangle^{- \frac32(1-\delta_\alpha)}\langle t\rangle^{\delta_\alpha} (\epsilon+\zeta(s,t)),
 \end{aligned}
 \end{equation}
 for $|\alpha|\le |\alpha_0|-1$. 
Note that these estimates are better than what's stated for $\partial_x^\alpha\hPsi_{s,t}(x,v)$. 

\subsubsection*{Bounds on $\partial_x^\alpha \hPsi^{tr}_{s,t}(x,v)$.}

Next, we give bounds on $\partial_x^\alpha\hPsi^{tr}_{s,t}(x,v)$, which we recall 
\begin{equation}\label{recalldaPtr}
\begin{aligned}
(t-s)\partial_x^\alpha\hPsi^{tr}_{s,t}(x,v)& = - \nabla_v \hv\int_s^t  
\partial_x^\alpha[ (\tau-s)  
Q^{tr} + Q^{tr,1}](\tau,X_{\tau,t}(x,v), V_{\tau,t}(x,v)) \; d\tau .
 \end{aligned}\end{equation}
We shall prove 
\begin{equation}\label{est-daPtr}
\begin{aligned}  
(t-s)\|\partial_x^\alpha\hPsi^{tr}_{s,t}\|_{L^\infty_{x,v}} 
 \lesssim 
 \epsilon\langle s\rangle^{- (1-\delta_\alpha)}\langle t\rangle^{\delta_\alpha} (\epsilon+\zeta(s,t)),
 \end{aligned}
 \end{equation}
for $|\alpha|\le |\alpha_0|-1$. Again, as $Q^{tr,1}(s)$ defined as in \eqref{def-Qtr1} is of the same form as that of $Q^{tr}(s)$ defined as in \eqref{def-Qtr}, following the proof of \eqref{est-daVtr} (precisely, starting from \eqref{est-Itrst}), we obtain 
\begin{equation}\label{bdPtr}
\begin{aligned}
\|(t-s)\partial_x^\alpha\hPsi^{tr}_{s,t} \|_{L^\infty_{x,v}} 
&\lesssim \epsilon\sum_{1\le |\mu|\le |\alpha|}  
\int_s^t (\tau-s) \langle \tau\rangle^{-3 +\frac32\delta_{\mu}+\epsilon_\mu} \prod_{1\le |\beta|\le |\alpha|} \| \partial_x^\beta X_{\tau,t}\|_{L^\infty_{x,v}}^{k_\beta} \; d\tau.
\end{aligned}
\end{equation}
In the case when $k_\beta = 0$ for all $|\beta| \ge 2$, we bound 
$$
\int_s^t  (\tau-s)\langle \tau\rangle^{-3 +\frac32\delta_{\mu}+\epsilon_\mu} \; d\tau
\le \langle t\rangle^{\delta_\alpha}\int_s^t  \langle \tau\rangle^{-2 +\delta_{\alpha}} \; d\tau \lesssim \langle t\rangle^{\delta_\alpha}\langle s\rangle^{-(1- \delta_\alpha)
} ,
$$
which verifies \eqref{est-daPtr} for this term. We next treat the case when $k_\beta \ge 1$ for some $|\beta|\ge 2$. Using \eqref{supprodX}, for each $|\mu|\le |\alpha|$, we bound 
$$
\begin{aligned}
\int_s^t  &(\tau-s)\langle \tau\rangle^{-3 +\frac32\delta_{\mu}+\epsilon_\mu} \prod_{2\le |\beta|\le |\alpha|} \| \partial_x^\beta X_{\tau,t}\|_{L^\infty_{x,v}}^{k_\beta} \; d\tau
\\
& \lesssim 
\int_s^t  \langle \tau\rangle^{-2 +\frac32\delta_{\mu}+\epsilon_\mu} \langle \tau\rangle^{-\frac12 \sum^*_\beta k_\beta+\frac{1}2 \sum^*_\beta\delta_\beta k_\beta}\langle t\rangle^{\sum^*_\beta\delta_\beta k_\beta} (\epsilon+\zeta(\tau,t))\; d\tau
\\
& \lesssim 
 \langle s\rangle^{-1 +\frac32\delta_{\mu}+\epsilon_\mu} \langle s\rangle^{-\frac12 \sum^*_\beta k_\beta+\frac12 \sum^*_\beta\delta_\beta k_\beta}\langle t\rangle^{\sum^*_\beta\delta_\beta k_\beta} (\epsilon+\zeta(s,t))
, \end{aligned}
$$
provided that the exponent of $\langle s\rangle$ remains strictly negative. 
Indeed, we shall check that 
\begin{equation}\label{bda} 
I_{\mu,\alpha} :=
 \langle s\rangle^{-1 +\frac32\delta_{\mu}+\epsilon_\mu} \langle s\rangle^{-\frac12 \sum^*_\beta k_\beta+\frac12 \sum^*_\beta\delta_\beta k_\beta}\langle t\rangle^{\sum^*_\beta\delta_\beta k_\beta}  \le  \langle s\rangle^{- (1-\delta_\alpha)}\langle t\rangle^{\delta_\alpha} 
.\end{equation}
Let $k_1$ be the number of $\partial_x X_{s,t}$ appearing in \eqref{bdPtr}. It then follows that 
$$|\mu| = \sum_\beta k_\beta = k_1 + \sum^*_\beta k_\beta, \qquad |\alpha| = \sum |\beta|k_\beta = k_1 + \sum^*_\beta |\beta|k_\beta,$$
with $\sum^*_\beta k_\beta\ge 1$ (and hence, $k_1 \le |\mu| - 1$). Since $\delta_\alpha$ is linear in $|\alpha|$, we get  
$$\delta_\mu= \delta_{k_1} + \sum^*_\beta \delta_{k_\beta}, \qquad \delta_\alpha= \delta_{k_1} + \sum^*_\beta \delta_\beta k_\beta.$$ 
Therefore, we bound  
$$
\begin{aligned}
 \langle s\rangle^{-\frac12 \sum^*_\beta k_\beta+\frac12 \sum^*_\beta\delta_\beta k_\beta}\langle t\rangle^{\sum^*_\beta\delta_\beta k_\beta}  
&=  \langle s\rangle^{-\frac{1}2(|\mu| - k_1) +\frac12(\delta_\alpha - \delta_{k_1})}\langle t\rangle^{\delta_\alpha - \delta_{k_1}} 
\\& = \langle s\rangle^{-\frac{1}2(|\mu| - k_1-1)} \langle s\rangle^{- \frac12 \delta_{k_1}}  \langle t\rangle^{- \delta_{k_1}}  \langle s\rangle^{-\frac{1}2(1-\delta_\alpha)}\langle t\rangle^{\delta_\alpha} 
\\& \le \langle s\rangle^{-\frac{1}2(|\mu| - k_1-1)} \langle s\rangle^{- \frac32 \delta_{k_1}}  \langle s\rangle^{-\frac{1}2(1-\delta_\alpha)}\langle t\rangle^{\delta_\alpha} ,
\end{aligned}$$
and since $\delta_\mu \le \delta_\alpha$, we bound $\langle s\rangle^{-1 +\frac32\delta_{\mu}+\epsilon_\mu} \le  \langle s\rangle^{-\frac12 +\delta_{\mu}+\epsilon_\mu} \langle s\rangle^{-\frac12(1 -\delta_\alpha)} $.
Therefore, \eqref{bda} would follow, provided that 
$$  \langle s\rangle^{-\frac{1}2(|\mu| - k_1-1)} \langle s\rangle^{- \frac32 \delta_{k_1}} \langle s\rangle^{-\frac12 +\delta_{\mu}+\epsilon_\mu} \le 
1$$
for $k_1\ge 0$ and $|\mu|\ge k_1 +1$. Since $\delta_{\mu}\le 1$, the inequality holds if $|\mu| \ge k_1 + 3$. In the case when $|\mu| = k_1 + 2$, we note that $\delta_\mu = \delta_{k_1} + 2\delta_1$, and so the inequality holds, since $2\delta_1 + \epsilon_\mu \le 3\delta_1 \le 1$, provided $|\alpha_0|\ge 4$.
Finally, in the remaining case when $|\mu| = k_1+1$, we note that $\delta_{\mu} = \delta_{k_1} + \delta_{1}$ and so the inequality again holds, since $\delta_1 + \epsilon_\mu \le 2 \delta_1 \le \frac12$, provided that $|\alpha_0|\ge 5$. This completes the proof of \eqref{est-daPtr}. 

\subsubsection*{Bounds on $\partial_x^\alpha \Psi^Q_{s,t}(x,v)$.}
Finally, we bound $\partial_x^\alpha \Psi^Q_{s,t}(x,v)$, which are computed by 
$$
\begin{aligned}
(t-s)\partial_x^\alpha \Psi^Q_{s,t}(x,v) &= \int_s^t \partial_x^\alpha V_{\tau,t}^Q(x,v)\; d\tau
\end{aligned}$$
where $V^Q_{s,t} (x,v) = \mathcal{O}(|V_{s,t}-v|^2)$. Using the Leibniz's formula for derivatives and the inductive bounds \eqref{inductV}, we compute 
$$
\begin{aligned}
|(t-s)\partial_x^\alpha \Psi^Q_{s,t}(x,v)| 
& \lesssim \sum_{|\beta|+|\mu| \le |\alpha|} \int_s^t |\partial_x^\beta (V_{s,t}-v)||\partial_x^\mu (V_{s,t}-v)| \; d\tau
\\
& \lesssim \epsilon \sum_{|\beta|+|\mu| \le |\alpha|} \int_s^t \langle \tau\rangle^{-\frac32 (1- \delta_\beta)-\frac32 (1- \delta_\mu)} \langle t\rangle^{\delta_\beta + \delta_\mu}(\epsilon+\zeta(\tau,t))\; d\tau.
\end{aligned}$$
Since $|\beta|+|\mu|\le |\alpha|$, we have $\delta_\beta + \delta_\mu \le \delta_\alpha$. Therefore, integrating in time the above integral yields 
\begin{equation}\label{est-daPQ}
\begin{aligned}
\|(t-s)\partial_x^\alpha \Psi^Q_{s,t}\|_{L^\infty_{x,v}} 
& \lesssim \epsilon \langle s\rangle^{-2 + \frac32 \delta_\alpha} \langle t\rangle^{\delta_\alpha}(\epsilon+\zeta(s,t)).
\end{aligned}
\end{equation}

Finally, combining all the estimates \eqref{est-daVosc}, \eqref{est-daVtr}, \eqref{est-daPosc}, \eqref{est-daPtr}, and \eqref{est-daPQ} into the expansions in \eqref{expPVst}, we complete the proof of the claim \eqref{claimXV1}, and hence the decay and boundedness in $L^\infty_{x,v}$ of the characteristics.


\subsection{Boundedness for top derivatives}\label{sec-Hdxchar}


In this section, we establish the boundedness of the top derivatives of the characteristics. Specifically, we prove the following proposition. 

\begin{proposition}\label{prop-HDBchar} 
Fix $|\alpha_0| \ge 4$. Let $X_{s,t}(x,v), V_{s,t}(x,v)$ be the nonlinear characteristic solving \eqref{ode-char}. Then, for $0\le s\le t$, there hold 
\begin{equation}\label{HsXV}
 \begin{aligned}
\sup_v \| \partial_x^{\alpha_0} X_{s,t}(x,v)\|_{L^2_x}  + \sup_v \| \partial_x^{\alpha_0} V_{s,t}(x,v)\|_{L^2_x}   &\lesssim \epsilon  \langle t\rangle^{\frac32 +\delta_{1}} ,
 \end{aligned}
 \end{equation}
 for $\delta_1 = \frac{1}{|\alpha_0|-1}$. 
\end{proposition}
\begin{proof} In this section, we bound the top derivatives of $X_{s,t}$ and $V_{s,t}$. Indeed, for $ |\alpha_0|\ge 4$, using again the expansion \eqref{expPVst}, we compute 
\begin{equation}\label{expXVtopda}
\begin{aligned}
\partial^{\alpha_0}_x V_{s,t}(x,v) &= -\partial_x^{\alpha_0} V^{osc}_{t,t}(x,v)+  \partial_x^{\alpha_0} V^{osc}_{s,t}(x,v)  +\partial^{\alpha_0}_xV^{tr}_{s,t}(x,v),
\\
\partial_x^{\alpha_0} \hPsi_{s,t}(x,v)& = -\partial_x^{\alpha_0} \hV^{osc}_{t,t}(x,v)+ \partial_x^{\alpha_0} \hPsi^{osc}_{s,t}(x,v)  +\partial_x^{\alpha_0} \hPsi^{tr}_{s,t}(x,v)+\partial_x^{\alpha_0} \Psi^Q_{s,t}(x,v) .
\end{aligned}
\end{equation}
The first term on the right hand side is already treated in \eqref{bdVosct}, namely 
\begin{equation}\label{L2daVosct}
\| \partial^{\alpha_0}_x V^{osc}_{t,t}\|_{L^2_{x}} \lesssim \| \partial^{\alpha_0}_x E^{osc,1}_{\pm}(t)\|_{L^2_{x}} \lesssim \|E^{osc}_{\pm}(t)\|_{H^{\alpha_0-1}_{x}}  \lesssim \epsilon .
\end{equation}
We shall now check the remaining terms. Indeed, recalling \eqref{daVosci}, we compute 
$$
\begin{aligned}  
\partial_x^{\alpha_0}V^{osc}_{s,t}= C_{\alpha}  
\partial_xE^{osc,1}_{\pm} (s) (\partial_x^{\alpha_0} X_{s,t})+ \sum_{1\le |\mu|\le |\alpha_0|} C_{\beta,\mu}  
\partial_x^\mu E^{osc,1}_{\pm} (s) \prod_{1\le |\beta|\le |\alpha_0|-1} (\partial_x^\beta X_{s,t})^{k_\beta} + \mathrm{l.o.t.}
\end{aligned}
$$
in which $\mathrm{l.o.t.}$ denotes terms that involve at least one derivative of $V_{s,t}$, which satisfy better estimates than those that involve $X_{s,t}$, see \eqref{decayXV} and \eqref{HsXV}. Therefore, ignoring the $\mathrm{l.o.t.}$ terms, we bound 
$$
\begin{aligned}  
\|\partial_x^{\alpha_0}V^{osc}_{s,t}\|_{L^2_x} 
&\lesssim 
\| \partial_xE^{osc,1}_{\pm} (s)\|_{L^\infty_{x,v}} \| \partial_x^{\alpha_0} X_{s,t}\|_{L^2_x}+ \sum_{1\le |\mu|\le |\alpha_0|} \| 
\partial_x^\mu E^{osc,1}_{\pm} (s)\|_{L^2_x} \prod_{1\le |\beta|\le |\alpha_0|-1} \| \partial_x^\beta X_{s,t}\|_{L^\infty_{x,v}}^{k_\beta}.
\end{aligned}
$$
Using \eqref{daVosci}, \eqref{bdVosct}, \eqref{daVosct}, and \eqref{decayXV}, since $|\beta|\le |\alpha_0|-1$, we bound 
$$ \| \partial_xE^{osc,1}_{\pm} (s)\|_{L^\infty_{x,v}} \| \partial_x^{\alpha_0} X_{s,t}\|_{L^2_x} \lesssim \epsilon \langle s\rangle^{-\frac32 (1- \delta_{1})}\| \partial_x^{\alpha_0} X_{s,t}\|_{L^2_x} $$
with $\delta_{1} = \frac{1}{|\alpha_0|-1}$, 
and 
$$
\begin{aligned}  
\prod_{1\le |\beta|\le |\alpha_0|-1} \| \partial_x^\beta X_{s,t}\|_{L^\infty_{x,v}}^{k_\beta}
&\lesssim  1 + \prod_{2\le |\beta|\le |\alpha_0|-1} \Big( \langle s\rangle^{-\frac12 (1- \delta_{\beta})} \langle t\rangle^{\delta_\beta}\Big)^{k_\beta}
\\
&\lesssim 1 +
\langle s\rangle^{-\frac12 \sum_* k_\beta +\frac12\sum_*\delta_{\beta}k_\beta} \langle t\rangle^{\sum_*\delta_\beta k_\beta}
\end{aligned}
$$
in which the summation $\sum_*$ is taken over $2\le |\beta| \le |\alpha_0|-1$. Recall that $\sum_\beta k_\beta = |\mu|$ and $\sum_\beta \delta_\beta k_\beta =\delta_{\alpha_0}.$ Note that $\sum_* k_\beta \ge 1$, since otherwise the estimated product is absent. In the case when $ \sum_* k_\beta \ge 2$, we bound 
$$ 
\langle s\rangle^{-\frac12 \sum_* k_\beta +\frac12\sum_*\delta_{\beta}k_\beta} \langle t\rangle^{\sum_*\delta_\beta k_\beta} \le 
\langle s\rangle^{-1 +\frac12 \delta_{\alpha_0}} \langle t\rangle^{\delta_{\alpha_0}} \le \langle t\rangle^{\delta_{\alpha_0}}, $$
since $\delta_{\alpha_0} = \frac{|\alpha_0|}{|\alpha_0|-1} \le 2$. On the other hand, in the remaining case when $\sum_* k_\beta =1$, there is a $\beta_0$, with $2\le |\beta_0|\le |\alpha_0|-1$, so that $k_{\beta_0}=1$ and $k_\beta=0$ for $\beta \not = \beta_0$. In this case, we compute 
$$ 
\langle s\rangle^{-\frac12 \sum_* k_\beta +\frac12\sum_*\delta_{\beta}k_\beta} \langle t\rangle^{\sum_*\delta_\beta k_\beta}  = \langle s\rangle^{-\frac12 +\frac12\delta_{\beta_0}} \langle t\rangle^{\delta_{\beta_0}} \le  \langle t\rangle^{\delta_{\alpha_0}} ,$$
since $\delta_{\beta_0} \le 1 \le \delta_{\alpha_0}$. This yields 
\begin{equation}\label{prodXtop}
\begin{aligned}  
\prod_{1\le |\beta|\le |\alpha_0|-1} \| \partial_x^\beta X_{s,t}\|_{L^\infty_{x,v}}^{k_\beta}
\lesssim 
\langle t\rangle^{\delta_{\alpha_0}}. 
\end{aligned}
\end{equation}
Therefore, 
\begin{equation}\label{est-topVosc}
\sup_v\|\partial_x^{\alpha_0}V^{osc}_{s,t}\|_{L^2_x} \lesssim \epsilon
\langle t\rangle^{\delta_{\alpha_0}} +  \epsilon \langle s\rangle^{-\frac32 (1- \delta_{1})}\sup_v\| \partial_x^{\alpha_0} X_{s,t}\|_{L^2_x} 
\end{equation}
Similarly, recalling \eqref{comPosc} and the fact that $E^{osc,2}(s)$ satisfies the same (or better) estimates than those for $E^{osc,1}_{\pm}(s)$, we obtain 
\begin{equation}\label{est-topPosc}
\sup_v\|(t-s)\partial_x^{\alpha_0}\hPsi^{osc}_{s,t}\|_{L^2_x} \lesssim \epsilon
\langle t\rangle^{\delta_{\alpha_0}} +  \epsilon \langle s\rangle^{-\frac32 (1- \delta_{1})}\sup_v\| \partial_x^{\alpha_0} X_{s,t}\|_{L^2_x} .
\end{equation}
Next, we bound $\partial_x^{\alpha_0} V^{tr}_{s,t}(x,v)$ and $\partial_x^{\alpha_0}\hPsi^{tr}_{s,t}(x,v)$. We shall prove that 
 \begin{equation}\label{est-topVPtr}
\begin{aligned}
\sup_v\|\partial_x^{\alpha_0}V^{tr}_{s,t}\|_{L^2_x} &\lesssim \epsilon
\langle t\rangle^{\delta_{\alpha_0}} + \epsilon \langle s\rangle^{-2 +\frac32\delta_{1}} \sup_{s\le \tau \le t}\sup_v\| \partial_x^{\alpha_0} X_{\tau,t}\|_{L^2_x}
\\
\sup_v\|(t-s)\partial_x^{\alpha_0}\hPsi^{tr}_{s,t}\|_{L^2_x} &\lesssim  \epsilon  \langle t\rangle^{\frac12 +\delta_{\alpha_0}}+ \epsilon \langle s\rangle^{-1 +\frac32\delta_{1}} \sup_{s\le \tau \le t}\sup_v\| \partial_x^{\alpha_0} X_{\tau,t}\|_{L^2_x}.
\end{aligned}
\end{equation}
In view of \eqref{recalldaVtr} and \eqref{recalldaPtr}, the latter is worse due to the double integrals in time. For this reason, we shall focus only on the following worst term, namely 
$$
\begin{aligned}
J_{s,t}& = \int_s^t  (\tau-s)
\Big[  C_\alpha 
\partial_xQ^{tr} (\tau) (\partial_x^{\alpha_0} X_{\tau,t}) + \sum_{2\le |\mu|\le |\alpha_0|} C_{\beta,\mu}  
\partial_x^\mu Q^{tr}(\tau) \prod_{1\le |\beta|\le |\alpha_0|-1} (\partial_x^\beta X_{\tau,t})^{k_\beta}\Big]\; d\tau.
 \end{aligned}$$
Recall from \eqref{dmuQtr} that $ \| \partial_x^\mu Q^{tr}(t)\|_{L^\infty_{x,v}} \lesssim \epsilon \langle t\rangle^{-3 +\frac32\delta_{\mu}+\epsilon_\mu} $. In addition, for $2\le |\mu|\le |\alpha_0|$, we compute 
$$
\begin{aligned}
\|  a_\pm(i\partial_x) \partial_x^\mu F(t)\star_x \phi_{\pm,1}\|_{L^2_{x}} &
\lesssim \| \partial_x^{\mu-1} S(t)\|_{L^2_{x}} \lesssim \epsilon \langle t\rangle^{-\frac32 + \delta_{\mu-1}}
 \end{aligned} 
$$
On the other hand, we have 
$$
\begin{aligned}
\| \partial_x^\mu(E\cdot  \nabla_x E^{osc,2}_\pm(t))\|_{L^2_{x}} 
& \lesssim 
\| E(t)\|_{L^\infty_x}\|E(t)\|_{H^{\mu}_x} \lesssim \epsilon^2 \langle t\rangle^{-\frac32}.
 \end{aligned} 
$$
That is,  recalling \eqref{comdaQtr}, we have 
\begin{equation}\label{bdsupvQtr}
\sup_v \| \partial_x^\mu Q^{tr}(t)\|_{L^2_{x}} \lesssim \epsilon \langle t\rangle^{-3/2 +\delta_{\mu-1}} .
\end{equation}
Therefore, together with \eqref{prodXtop}, we compute 
$$
\begin{aligned}
\|J_{s,t}\|_{L^2_x}
& \lesssim \int_s^t  (\tau-s)
\Big[ 
\|\partial_xQ^{tr} (\tau) \|_{L^\infty_x}\| \partial_x^{\alpha_0} X_{\tau,t}\|_{L^2_x} + \sum_{2\le |\mu|\le |\alpha_0|}  
\|\partial_x^\mu Q^{tr}(\tau)\|_{L^2_x} \prod_{1\le |\beta|\le |\alpha_0|-1} \|\partial_x^\beta X_{\tau,t}\|_{L^\infty_x}^{k_\beta}\Big]\; d\tau
 \end{aligned}$$
in which we bound 
$$
\begin{aligned}
\int_s^t  (\tau-s)
\|\partial_xQ^{tr} (\tau) \|_{L^\infty_x}\| \partial_x^{\alpha_0} X_{\tau,t}\|_{L^2_x} 
& \lesssim \epsilon \int_s^t  (\tau-s)
\langle \tau\rangle^{-3 +\frac32\delta_{1}+\epsilon_1} \| \partial_x^{\alpha_0} X_{\tau,t}\|_{L^2_x} \; d\tau 
\\
& \lesssim \epsilon \langle s\rangle^{-1 +2\delta_{1}} \sup_{s\le \tau \le t}\sup_v\| \partial_x^{\alpha_0} X_{\tau,t}\|_{L^2_x},
 \end{aligned}$$
noting that $\epsilon_1 \le \frac12\delta_1$ and $2\delta_1< 1$, provided that $|\alpha_0|\ge 3$. On the other hand, for $2\le |\mu|\le |\alpha_0|$, we bound 
$$
\begin{aligned}
\int_s^t & (\tau-s)
\|\partial_x^\mu Q^{tr}(\tau)\|_{L^2_x} \prod_{1\le |\beta|\le |\alpha_0|-1} \|\partial_x^\beta X_{\tau,t}\|_{L^\infty_x}^{k_\beta}
\; d\tau
\\&
\lesssim \epsilon \int_s^t  \langle \tau\rangle^{-\frac12 +\delta_{\mu-1}} \Big[ 1+
\langle \tau\rangle^{-\frac12 \sum_* k_\beta +\frac12\sum_*\delta_{\beta}k_\beta} \langle t\rangle^{\sum_*\delta_\beta k_\beta} \Big]
\; d\tau
\\&
\lesssim \epsilon  \langle t\rangle^{\frac12 +\delta_{\alpha_0}}  + \epsilon
 \int_s^t  \langle \tau\rangle^{-\frac12 +\delta_{\mu-1}} 
\langle \tau\rangle^{-\frac12 \sum_* k_\beta +\frac12\sum_*\delta_{\beta}k_\beta} \langle t\rangle^{\sum_*\delta_\beta k_\beta}
\; d\tau
 \end{aligned}$$
in which the summation $\sum_*$ is taken over $2\le |\beta| \le |\alpha_0|-1$. Note that in the summation, $|\mu|\le |\alpha_0|-1$, since the case $|\mu|=|\alpha_0|$ corresponds to the case $|\beta| =1$ and $k_\beta = |\alpha_0|$ (i.e. no other derivatives appear other than the first derivatives). We shall prove that 
\begin{equation}\label{intstar}
 \int_s^t  \langle \tau\rangle^{-\frac12 +\delta_{\mu-1}} 
\langle \tau\rangle^{-\frac12 \sum_* k_\beta +\frac12\sum_*\delta_{\beta}k_\beta} \langle t\rangle^{\sum_*\delta_\beta k_\beta}
\; d\tau  \lesssim  \langle t\rangle^{\frac12 +\delta_{\alpha_0}} ,
\end{equation}
for $2\le |\mu|\le |\alpha_0|-1$. Recall that 
$\sum_\beta k_\beta = |\mu|$ and $\sum_\beta |\beta| k_\beta =|\alpha_0|.$ 
Therefore, 
$$|\mu|= k_1+ \sum_* k_\beta, \qquad  |\alpha_0|= k_1 + \sum_* |\beta| k_\beta,$$ 
where $k_1$ is the number of $\partial_x X_{s,t}$ appearing in the time integral (noting that $\partial_x^{\alpha_0}X_{s,t}$ has already been treated). Noting $\delta_{\alpha_0} = 1 + \delta_1$, we thus have 
$$
\begin{aligned}
\langle \tau\rangle^{-\frac12 +\delta_{\mu-1}}  \langle \tau\rangle^{-\frac12 \sum_* k_\beta+\frac12 \sum_*\delta_\beta k_\beta}\langle t\rangle^{\sum_*\delta_\beta k_\beta}  
&= \langle \tau\rangle^{-\frac12 +\delta_{\mu-1}}  \langle \tau\rangle^{-\frac{1}2(|\mu| - k_1) +\frac12(\delta_{\alpha_0} - \delta_{k_1})}\langle t\rangle^{\delta_{\alpha_0} - \delta_{k_1}} 
\\
&= \langle \tau\rangle^{\delta_{\mu-1}-\frac{1}2(|\mu| - k_1) -\frac12( \delta_{k_1}-\delta_1)}\langle t\rangle^{\delta_{\alpha_0}- \delta_{k_1}}
\\
&\le \langle \tau\rangle^{\delta_{\mu-1} - \delta_{k_1}-\frac{1}2(|\mu| - k_1-1) -\frac12(\delta_{k_1} - \delta_1)} \langle \tau\rangle^{-\frac12}\langle t\rangle^{\delta_{\alpha_0}}.
\end{aligned}$$
Thus, \eqref{intstar} follows, provided that 
$$
\delta_{\mu-1} - \delta_{k_1}-\frac{1}2(|\mu| - k_1-1) -\frac12(\delta_{k_1} - \delta_1) \le 0,$$
for $k_1+1 \le |\mu|\le |\alpha_0| - 1$. Since $\delta_1\le \frac13$ and $ \delta_{\mu-1} \le 1$, the inequality is clear if $|\mu| - k_1 \ge 3$. On the other hand, in the case when $|\mu|-k_1= 2$, we note $\delta_{\mu-1} - \delta_{k_1} =\delta_1$, and hence the inequality holds, since $\frac32 \delta_1 \le \frac12$, provided $|\alpha_0|\ge 4$. Finally, in the case when $|\mu|-k_1 = 1$, the inequality again holds, since $\delta_{\mu-1} - \delta_{k_1}=0$ and $\delta_{k_1} - \delta_1\ge0$ (noting $k_1 \ge 1$, since $|\mu| = k_1 + 1\ge 2$). This gives \eqref{intstar}, and therefore completes the proof of \eqref{est-topVPtr}. 

Finally, we bound $\partial_x^\alpha \Psi^Q_{s,t}(x,v)$. Indeed, by definition (see Proposition \ref{prop-charPsi}) and the decay estimates \eqref{decay-Vst}, we compute 
$$
\begin{aligned}
\|(t-s)\Psi^Q_{s,t} \|_{H^{\alpha_0}_x} 
& \lesssim  \int_s^t \|V_{\tau,t}-v\|_{L^\infty_x}\| V_{\tau,t}-v \|_{H^{\alpha_0}_x} \; d\tau
\\
& \lesssim \epsilon\int_s^t \langle \tau\rangle^{-3/2} \| V_{\tau,t}-v \|_{H^{\alpha_0}_x}\; d\tau
\\
& \lesssim \epsilon\langle s\rangle^{-1/2} 
\sup_{s\le \tau \le t}\sup_v\| V_{\tau,t}-v \|_{H^{\alpha_0}_x}.
\end{aligned}$$
Combining \eqref{est-topVosc}, \eqref{est-topPosc}, \eqref{est-topVPtr}, and the above estimates into the expansion \eqref{expXVtopda}, and recalling that $\partial_x^{\alpha_0} X_{s,t} = (t-s)\partial_x^{\alpha_0}\hPsi_{s,t}$, we obtain 
$$
\begin{aligned} 
\|\partial_x^{\alpha_0} V_{s,t}\|_{L^2_x} 
& \lesssim \epsilon
\langle t\rangle^{\delta_{\alpha_0}} +  \epsilon \langle s\rangle^{-\frac32 (1- \delta_{1})}\sup_v\| \partial_x^{\alpha_0} X_{s,t}\|_{L^2_x} 
\\
\|\partial_x^{\alpha_0} X_{s,t}\|_{L^2_x} &\lesssim \epsilon  \langle t\rangle^{\frac12 +\delta_{\alpha_0}}+ \epsilon \langle s\rangle^{-1 +\frac32\delta_{1}} \sup_{s\le \tau \le t}\sup_v\| \partial_x^{\alpha_0} X_{\tau,t}\|_{L^2_x} + \epsilon\langle s\rangle^{-1/2} 
\sup_{s\le \tau \le t}\sup_v\| V_{\tau,t}-v \|_{H^{\alpha_0}_x}.
\end{aligned}$$
Adding up the two inequalities and taking $\epsilon$ sufficiently small to absorb the top derivatives to the left, we obtain \eqref{HsXV}, and thus complete the proof of Proposition \ref{prop-HDBchar}. 
\end{proof}


\section{Density estimates}\label{sec-sourceest}


In this section, we prove decay estimates on the density. Precisely, we obtain the following. 

\begin{proposition}\label{prop-bdS} Fix $\alpha_0\ge 8$ as in \eqref{bootstrap-Hs}. 
Let $S(t,x)$ be the source density defined by 
\begin{equation}\label{defS}
\begin{aligned}
S(t,x) &= \int_{\RR^3} f_0(X_{0,t}(x,v) , V_{0,t}(x,v)) \varphi(v)\, dv,
\end{aligned}\end{equation}
where $\varphi(v)$ is some smooth and bounded function in $v$. Then, for any $1 \le p \le \infty$, there hold
\begin{equation}\label{Lpbounds-cS0}
\begin{aligned}
\|S(t)\|_{L^p_x} + \|\partial_xS(t)\|_{L^p_x} &\lesssim \epsilon_0  \langle t\rangle^{-3(1-1/p)} ,
\\
\|\partial_x^\alpha S(t)\|_{L^p_x}   &\le \epsilon \langle t\rangle^{-3(1-1/p) + \delta_\alpha} 
,\end{aligned}\end{equation}
for $|\alpha|\le |\alpha_0|-1$,  with $\delta_\alpha = \frac{|\alpha|}{|\alpha_0|-1}$. In addition, 
\begin{equation}\label{Hsbounds-cS0}
\begin{aligned}
\|S(t)\|_{H^{\alpha_0}_x} &\lesssim \epsilon \langle t\rangle^{\delta_1},
\end{aligned}
\end{equation}
with $\delta_1 = \frac{1}{|\alpha_0|-1}$.

\end{proposition}

\begin{proof} Introducing the change of variables $y = X_{0,t}(x,v)$ in the integral \eqref{defS}, we write
$$
\begin{aligned}
S(t,x) & = \int_{\RR^3} f_0(X_{0,t}(x,v) , V_{0,t}(x,v)) \varphi(v)\, dv
= \int_{\RR^3} f_0(y, V_{0,t}(x,v)) \varphi(v)\frac{dy}{ \det (\nabla_v X_{0,t}(x,v))}
\end{aligned}
$$
in which $v = v_{0,t}(x,y)$, the inverse map of $v\mapsto y = X_{0,t}(x,v)$. Using \eqref{Jacobian-dxdv} and the fact that $\varphi(v)$ is bounded in $v$, we thus obtain 
$$
\begin{aligned}
|S(t,x)| &\lesssim t^{-3}  \int_{\RR^3} |f_0(y, V_{0,t}(x,v_{0,t}(x,y))) |  \langle v_{0,t}\rangle^5\; dy
\\& \lesssim t^{-3}\int_{\RR^3} \sup_v \langle v\rangle^5 |f_0(y, v)| dy
\end{aligned}
$$
in which we noted that $\langle v\rangle \le 2 \langle V_{0,t}(x,v)\rangle$, using $\| V_{0,t} - v\|_{L^\infty_{x,v}} \lesssim \epsilon$. This gives the $L^\infty$ estimate. On the other hand, since $(x,v)\mapsto (X_{0,t},V_{0,t})$ is a volume-preserving map, we have 
$$
\begin{aligned}
\|S(t)\|_{L^1_x} &\lesssim \iint |f_0(X_{0,t}(x,v) , V_{0,t}(x,v))| \; dxdv = \| f_0\|_{L^1_{x,v}} .
\end{aligned}
$$
The $L^p$ bounds in \eqref{Lpbounds-cS0} follow from the interpolation between $L^1$ and $L^\infty$ spaces. 
Similarly, we compute  
$$
\begin{aligned}
\partial_x S(t,x) &= \int_{\RR^3} \partial_xX_{0,t} \cdot \nabla_x f_0(X_{0,t}(x,v) , V_{0,t}(x,v)) \varphi(v)\, dv
\\
&\quad + \int_{\RR^3}  \partial_xV_{0,t} \cdot \nabla_v f_0(X_{0,t}(x,v) , V_{0,t}(x,v)) \varphi(v)\, dv .
\end{aligned} 
$$  
Using Proposition \ref{prop-Dchar}, we have the boundedness of both 
$ \partial_x X_{0,t}$ and $\partial_xV_{0,t} $. 
Therefore, the estimates for $\partial_x S(t,x)$ follow similarly as done above for $S(t,x)$. 

Finally, we prove the decay and boundedness for high derivatives of $S(t)$. For simplicity, we take $\varphi(v)=1$ in \eqref{defS}. Using the Faa di Bruno's formula
for derivatives of a composite function, see \eqref{daVosci}, we compute 
\begin{equation}\label{daS}
\begin{aligned}  
\partial_x^\alpha S(t,x) & = \sum_{1\le |\mu|\le |\alpha|} C_{\beta,\mu}  \int_{\RR^3} 
\partial_x^\mu f_0(X_{0,t}, V_{0,t})\prod_{1\le |\beta|\le |\alpha|} (\partial_x^\beta X_{0,t})^{k_\beta} \; dv + \mathrm{l.o.t.}
\end{aligned}\end{equation}
in which $\mathrm{l.o.t.}$ denotes terms involving derivatives of $\partial_x^\beta V_{0,t}$, which are better than those for $\partial_x^\beta X_{0,t}$. Consider the case when $|\alpha|\le |\alpha_0|-1$. In this case, we may use the decay and boundedness of $\partial_x^\beta X_{0,t}$ from Proposition \ref{prop-HDchar} to bound 
$$ 
\begin{aligned}
\int_{\RR^3} 
\partial_x^\mu f_0(X_{0,t}, V_{0,t})\prod_{1\le |\beta|\le |\alpha|} (\partial_x^\beta X_{0,t})^{k_\beta} \; dv
&\lesssim 
\int_{\RR^3} 
|\partial_x^\mu f_0(X_{0,t}, V_{0,t})| \prod_{2\le |\beta| \le |\alpha|} \langle t\rangle^{\delta_\beta}
\; dv
\\&\lesssim \langle t\rangle^{\delta_\alpha}
\int_{\RR^3} 
|\partial_x^\mu f_0(X_{0,t}(x,v), V_{0,t})| \; dv
\end{aligned}
$$
in which we have used $\sum_\beta \delta_\beta k_\beta = \delta_\alpha$. Now following the same proof as done above, we obtain the decay of $\int_{\RR^3} 
|\partial_x^\mu f_0(X_{0,t}(x,v), V_{0,t})| \; dv$
in $L^p_x$ as desired. Finally, for $|\alpha| = |\alpha_0|$, we write 
$$ 
\begin{aligned}
&\sum_{1\le |\mu|\le |\alpha|} \int_{\RR^3} 
\partial_x^\mu f_0(X_{0,t}, V_{0,t})\prod_{1\le |\beta|\le |\alpha_0|} (\partial_x^\beta X_{0,t})^{k_\beta} \; dv 
\\&= \int_{\RR^3} 
\partial_xf_0(X_{0,t}, V_{0,t})(\partial_x^{\alpha_0} X_{0,t}) \; dv 
+\sum_{2\le |\mu|\le |\alpha|}  \int_{\RR^3} 
\partial_x^\mu f_0(X_{0,t}, V_{0,t})\prod_{1\le |\beta|\le |\alpha_0|-1} (\partial_x^\beta X_{0,t})^{k_\beta} \; dv.
\end{aligned}
$$
The last integral term involves $\partial_x^\beta X_{0,t}$ with $|\beta|\le |\alpha_0|-1$, and is therefore treated as done above, using Proposition \ref{prop-HDchar}. As for the first term, we bound  
$$ 
\begin{aligned}
\int_{\RR^3} 
\partial_xf_0(X_{0,t}, V_{0,t})(\partial_x^{\alpha_0} X_{0,t}) \; dv
&\lesssim \Big(
\int_{\RR^3} 
|\partial_xf_0(X_{0,t}, V_{0,t})| \; dv\Big)^{1/2} \Big( \int_{\RR^3} 
\partial_xf_0(X_{0,t}, V_{0,t})|\partial_x^{\alpha_0} X_{0,t}|^2 \; dv\Big)^{1/2}
\\
&\lesssim \epsilon_0 \langle t\rangle^{-3/2}\Big( \int_{\RR^3} 
\partial_xf_0(X_{0,t}, V_{0,t})|\partial_x^{\alpha_0} X_{0,t}|^2 \; dv\Big)^{1/2}.\end{aligned}
$$
Therefore, using Proposition \ref{prop-HDBchar} and the fact that $\|\langle v\rangle^4\partial_xf_0\|_{L^\infty_{x,v}} \lesssim 1$, we bound 
$$ 
\begin{aligned}
\Big\|
\int_{\RR^3} 
\partial_xf_0(X_{0,t}, V_{0,t})(\partial_x^{\alpha_0} X_{0,t}) \; dv\Big\|^2_{L^2_x}
&\lesssim \epsilon^2_0 \langle t\rangle^{-3}\iint_{\RR^3\times \RR^3} 
\partial_xf_0(X_{0,t}, V_{0,t})|\partial_x^{\alpha_0} X_{0,t}|^2 \; dxdv
\\&\lesssim \epsilon^2_0 \langle t\rangle^{-3}\int_{\RR^3} 
\langle v\rangle^{-4} \sup_v\|\partial_x^{\alpha_0} X_{0,t}\|_{L^2_x}^2 dv.\\&\lesssim \epsilon^2_0 \langle t\rangle^{2\delta_1},
\end{aligned}
$$
proving \eqref{Hsbounds-cS0}. This completes the proof of the proposition. 
\end{proof}


\section{Nonlinear oscillatory field}\label{sec-decayosc}


In this section, we bound the oscillatory electric field. Precisely, the main result of this section is the following proposition.

\begin{proposition}\label{prop-decayEosc} Fix $\alpha_0\ge 8$ as in \eqref{bootstrap-Hs}. 
Let $S$ be the source density defined as in \eqref{defS}. Then, there hold
\begin{equation}\label{decayEosc}
\begin{aligned}
\|G_\pm^{osc} \star_{t,x} \nabla_xS(t)\|_{L^\infty_x} &\lesssim (\epsilon_0 + \epsilon^2) \langle t\rangle^{-3/2} 
\end{aligned}\end{equation}
and 
\begin{equation}\label{bdEosc}
\begin{aligned}
\|G_\pm^{osc} \star_{t,x} \nabla_x S(t)\|_{H^{\alpha_0}_x} &\lesssim (\epsilon_0 + \epsilon^2)
\\
\|G_\pm^{osc} \star_{t,x} \nabla_x S(t)\|_{H^{\alpha_0+1}_x} &\lesssim (\epsilon_0 + \epsilon^2) \langle t\rangle^{\delta_1} .
\end{aligned}\end{equation}

\end{proposition}

In view of Proposition \ref{prop-bdS}, the source density decays at rate of order $t^{-3}$ in $L^\infty_x$. Clearly, this is insufficient to establish the Klein-Gordon type decay of the oscillatory field $G_\pm^{osc} \star_{t,x} S(t)$ through the spacetime convolution, yielding only a decay of order $t^{-1/2}$. We need to further exploit the nonlinear structure of the oscillatory field.


\subsection{Some useful lemmas}\label{sec-lemmas}

In this section, we give some useful estimates on the oscillatory Green kernel. Precisely, we obtain the following. 

\begin{lemma}\label{lem-decayoscS0}
Let $G^{osc}_\pm(t,x)$ be the Klein-Gordon oscillatory Green function. Then, for $\alpha \in \{0,1\}$, there hold
\begin{equation}\label{Gosc-l2}
\|\langle i\partial_x \rangle^{-\alpha}G^{osc}_\pm\star_x  \nabla_x f\|_{L^2_x} \lesssim \| f\|_{L^2_x}, \qquad \|\langle i\partial_x \rangle^{-\alpha}G^{osc}_\pm\star_x  \nabla_x f\|_{L^\infty_x} \lesssim \| f\|_{H^{-\alpha + \frac{3}{2}+}_x}.
\end{equation}
In addition, for $p\in [2,\infty]$, there hold
\begin{equation}\label{Gosc-lp}
\begin{aligned}
\|\langle i\partial_x \rangle^{-\alpha}G^{osc}_\pm\star_x \nabla_x f\|_{L^p_x} & \lesssim \langle t\rangle^{-3\left(\frac 12-\frac1p\right)}\| f \|_{W_x^{1-\alpha+\lfloor 3\left(1-\frac{2}{p}\right) \rfloor, p'}} 
\end{aligned}
\end{equation}
with $\frac 1p+\frac 1{p'}=1$ and $\lfloor s \rfloor$ denoting the largest integer that is smaller than $s$. 
\end{lemma}

\begin{proof} The $L^2$ estimate is immediate, recalling that $\FG^{osc}_\pm(t,k) = e^{\pm i \langle k\rangle t} a_\pm(k)$, where $a_\pm(k)$ is sufficiently smooth and satisfying $a_\pm(k) \lesssim \langle k\rangle^{-1}$, while the $L^\infty$ estimate follows from the bound 
$$\| g\|_{L^\infty_x} \le \| \Fg\|_{L^1_k} \lesssim\| \langle k\rangle^{\frac32+}\Fg\|_{L^2_k} .$$ 
As for the $L^p$ estimates, we recall the following standard dispersive estimates, see \cite{RS, HKNR4},  
$$\|\langle i\partial_x \rangle^{-\alpha}G^{osc}_\pm\star_x \nabla_x f\|_{L^p_x} \lesssim \langle t\rangle^{-3\left(\frac 12-\frac1p\right)}\| \langle i\partial_x \rangle^{-\alpha} a_\pm(i\partial_x)\nabla_x f\|_{B^{3\left(1-\frac{2}{p}\right)}_{p',2}}$$
with $\frac 1p+\frac 1{p'}=1$, with $\|\cdot \|_{B^{s}_{p,q}}$ denoting the standard Besov spaces, see Section \ref{sec-FM}. In particular, using \eqref{eq:fouriermultbesov}, for any $s\ge 0$, we have
$$
\|   \langle i\partial_x \rangle^{-\alpha}  a(i\partial_x) \nabla_x  f \|_{B^{s}_{p',2}} \lesssim \|  f\|_{W^{1 - \alpha+ \lfloor s \rfloor, p'}},
$$
in which we note that in the case $\alpha =1$, we used $s -  \lfloor s \rfloor < \alpha$. 
Combining the two estimates, we obtain the lemma. 
\end{proof}

\begin{lemma}\label{lem-decayosc}
Let $\alpha \in \{0,1\}$. Suppose that 
$$ \int_0^t \| J(s)\|_{W_x^{1-\alpha+\lfloor 3\left(1-\frac{2}{p}\right) \rfloor, p'}} \;ds \lesssim 1, \qquad \| J(t)\|_{H^{-\alpha +\frac{3}{2}+}_x} \lesssim \langle t\rangle^{-5/2} .$$
Then, there holds 
$$ \| \langle i\partial_x \rangle^{-\alpha}G^{osc}_\pm \star_{t,x} \nabla_x J(t)\|_{L^p_x} \lesssim \langle t\rangle^{-3(1/2-1/p)} $$
for $2\le p\le\infty$. 
\end{lemma}

\begin{proof}
The lemma is direct. Indeed, using Lemma \ref{lem-decayoscS0}, we compute 
$$
\begin{aligned} 
&\| \langle i\partial_x \rangle^{-\alpha}G^{osc}_\pm \star_{t,x} \nabla_x J\|_{L^p_x} 
\\&\le \int_0^t \| \langle i\partial_x \rangle^{-\alpha}G^{osc}_\pm(t-s) \star_x \nabla_x J(s)\|_{L^p_x} \; ds
\\& \lesssim \int_0^{t/2} \langle t-s\rangle^{-3(1/2-1/p)} \| J(s)\|_{W_x^{1-\alpha+\lfloor 3\left(1-\frac{2}{p}\right) \rfloor, p'}}\; ds + \int_{t/2}^t \| J(s)\|_{H^{-\alpha +\frac{3}{2}+}_x} \; ds
\\& \lesssim \langle t\rangle^{-3(1/2-1/p)}\int_0^{t/2}\| J(s)\|_{W_x^{1-\alpha+\lfloor 3\left(1-\frac{2}{p}\right) \rfloor, p'}}\; ds +\int_{t/2}^t \langle s\rangle^{-5/2}\; ds,
\end{aligned}$$
which ends the proof, using the assumptions made on $J(t)$. 
\end{proof}

\begin{remark}
Observe that deriving the Klein-Gordon dispersive decay \eqref{Gosc-lp} requires a loss of derivatives. However, the number of derivative losses does not affect the nonlinear iterative scheme devised in this paper, upon requiring a higher order of derivatives at the top order.  
\end{remark}

\subsection{Collective oscillations}

In this section, we prove the following result, which plays an important role in the nonlinear iterative scheme. 

\begin{proposition}\label{prop-convGoscS0}
Let $S(t,x)$ be the source density defined as in \eqref{defS}. Then, there holds
\begin{equation}\label{convGoscS0}
G^{osc}_\pm\star_{t,x}  S(t,x)  = - G^{osc}_\pm(t,x)\star_x S_{\pm,0}(x) + a_\pm(i\partial_x) S_{\pm,1}(t,x) + G_\pm^{osc} \star_{t,x} S_{\pm,2}(t,x)
\end{equation}
where $a_\pm(k) = \frac{ \mp i }{2\langle k\rangle}$, and 
\begin{equation}\label{def-Spm12}
\begin{aligned}
S_{\pm,0}(x)&=  \int \phi_{\pm,1} (x,v)\star_x f_0(x,v) \varphi(v)\,dv 
\\
S_{\pm,1}(t,x)  &=  \int \phi_{\pm,1}(x,v)\star_x f_0(X_{0,t}(x,v) , V_{0,t}(x,v)) \varphi(v)\,dv 
\\
S_{\pm,2}(t,x)  &= -
 \int \nabla_v\phi_{\pm,1}(x,v) \star_x \Big[ E(t,x)f_0(X_{0,t}(x,v) , V_{0,t}(x,v)) \Big]  \varphi(v)\, dv 
 \end{aligned}
 \end{equation}
 recalling $\phi_{\pm,1}(x,v)$ defined as in \eqref{def-phipmj}.
\end{proposition}

\begin{proof} By definition, we write
$$
\begin{aligned}
G^{osc}_\pm\star_{t,x}  S(t,x)
&= \int_0^t \iint G^{osc}_\pm(t  - \tau,x - y)  f_0(X_{0,\tau}(y,v) , V_{0,\tau}(y,v)) \varphi(v) dvdy ds. 
\end{aligned}$$
We introduce the change of variables  
$$
(y,v) \quad \longmapsto \quad (Y,V) = ( X_{0,\tau}(y,v), V_{0,\tau}(y,v)), 
$$ 
whose Jacobian determinant is equal to one, recalling that $(X_{s,t}, V_{s,t})$ are the nonlinear characteristic curves of the divergence-free vector field $(\hv, E(t,x))$ in the phase space $\RR^3_x\times \RR^3_v$. By the time reversibility, we can write 
$$
(y,v) = (X_{\tau,0}(Y,V), V_{\tau,0}(Y,V)) .
$$
This leads to 
$$
\begin{aligned}
G^{osc}_\pm\star_{t,x}  S(t,x)
&= \int_0^t \iint G^{osc}_\pm(t  - \tau,x - y)  f_0(X_{0,\tau}(y,v) , V_{0,\tau}(y,v)) \varphi(v) dvdy ds
\\
&= \iint \Bigl[ \int_0^t G^{osc}_\pm(t  - \tau,x - X_{\tau,0}(Y,V)) \varphi(V_{\tau,0}(Y,V))
\, d\tau \Bigr] f_0(Y,V) \, dY dV.
\end{aligned}$$
Recall from \eqref{def-Gr} that the Fourier transform  $G^{osc}_\pm(t,x)$ is of the form $\FG_\pm^{osc}(t,k) = e^{\lambda_\pm(k)t} a_\pm(k)$. Therefore, similar to the calculation done in \eqref{int-Eosc}, we compute the integral of $G^{osc}_\pm(t,x)$ over the nonlinear characteristic curves, yielding 
$$
\int_0^t G^{osc}_\pm(t-\tau,x - Y - (\tau - s) \hv) d\tau = G^{osc,1}_\pm(0,x - Y - t  \hv,v) 
- G^{osc,1}_\pm(t , x - Y,v) ,
$$
where, as in \eqref{def-Eoscj}-\eqref{def-phipmj}, we denote $G^{osc,j}_\pm(t,x,v) = [G^{osc}_{\pm} \star_{x} \phi_{\pm,j}](t,x,v)$. This leads to 
$$
G^{osc}_\pm\star_{t,x}  S(t,x) = J_1(t,x)  + J_2(t,x)   + \cE^{R}(t,x) $$
where
$$
\begin{aligned}
J_1(t,x)  &=  \iint G^{osc,1}_\pm(0, x - X_{t,0}(Y,V),V_{t,0}(Y,V)) \varphi(V_{t,0}(Y,V))f_0(Y,V)  \, dY dV 
\\
J_2(t,x)  &= -  \iint G^{osc,1}_\pm(t,  x - Y,V)  \varphi(V) f_0(Y,V)  \, dY dV 
\\
\cE^R(t,x)  &= - \int_0^t \iint E(\tau, X_{\tau,0})\cdot \nabla_v(G^{osc,1}_{\pm}(t-\tau,x-X_{\tau,0},V_{\tau,0}) \varphi(V_{\tau,0}) )  f_0(Y,V) 
\, dV dYd \tau.
\end{aligned}$$ 
Reversing the change of variables $(y,v) = (X_{\tau,0}(Y,V),V_{\tau,0}(Y,V))$, we obtain the proposition.  
\end{proof}

\subsection{Decay estimates}

We are now ready to prove Proposition \ref{prop-decayEosc} by estimating each term in the representation \eqref{convGoscS0} of $G^{osc}_\pm \star_{t,x}  S(t,x)$. Indeed, we recall that 
\begin{equation}\label{GSosc1}
\begin{aligned}
G^{osc}_\pm\star_{t,x}  \nabla_x S(t,x)  &= - G^{osc}_\pm(t,\cdot)\star_x  \nabla_x S_{\pm,0}(x) + a_\pm(i\partial_x)  \nabla_x S_{\pm,1}(t,x) 
\\&\quad + G_\pm^{osc} \star_{t,x}  \nabla_x S_{\pm,2}(t,x),
\end{aligned}\end{equation}
in which we note that the second term contributes into $E^r(t,x)$. 
Let us estimate term by term in \eqref{GSosc1}. 

\subsubsection*{Source term $S_{\pm,0}$.}
First, using Lemma \ref{lem-decayoscS0}, we bound 
$$
\begin{aligned}
\|G^{osc}_\pm(t,\cdot)\star_x  \nabla_x S_{\pm,0}(x)\|_{L^\infty_x} &\lesssim t^{-3/2}\| S_{\pm,0}\|_{W^{4,1}_x}
\\
\|G^{osc}_\pm(t,\cdot)\star_x  \nabla_x S_{\pm,0}(x)\|_{H^{\alpha_0+1}_x} &\lesssim \| S_{\pm,0}\|_{H^{\alpha_0+1}_x}.
\end{aligned}$$
Recall that by definition, 
$$S_{\pm,0}(x)=  \int \phi_{\pm,1} (x,v)\star_x f_0(x,v) \varphi(v)\,dv .$$
Therefore, using \eqref{Lp-convphi}, we bound 
$$\| S_{\pm,0}\|_{W^{4,1}_x} \lesssim \int \|  f_0 (\cdot,v) \|_{W^{4,1}_x} \varphi(v) \; dv \le C_0 \| f_0\|_{L^1_v W^{4,1}_x} .$$
Similarly, we have $\| S_{\pm,0}\|_{H^{\alpha_0+1}_x} \lesssim  \| f_0\|_{L^1_v H^{\alpha_0+1}_x} .$ This proves the desired estimates stated in Proposition \ref{prop-decayEosc} for the first term in \eqref{GSosc1}. 

\subsubsection*{Source term $S_{\pm,1}$.}

Next we treat the second term in \eqref{GSosc1}, namely the source term $a_\pm(i\partial_x)  \nabla_x S_{\pm,1}(t,x)$. Recall that 
$$S_{\pm,1}(t,x)  =  \int \phi_{\pm,1}(x,v)\star_x f_0(X_{0,t}(x,v) , V_{0,t}(x,v)) \varphi(v)\,dv .
$$
Next, using the expansion \eqref{expand-phi1}, we may write 
\begin{equation}\label{seriesn}
\begin{aligned}
S_{\pm,1}(t,x)  &= \sum_{n\ge 0} a_{\pm,n}(i\partial_x) :: S_{1,n}(t,x)
 \end{aligned}\end{equation}
where $a_{\pm,n}(k) =  \mp i \langle k\rangle^{-1} (\pm 1)^n\langle k\rangle^{-n} k^{\otimes n}$ are smooth Fourier multipliers that satisfy $a_{\pm,n}(k) \le \langle k\rangle^{-1}$, and  
$$ S_{1,n}(t,x) =  
 \int f_0(X_{0,t}(x,v) , V_{0,t}(x,v)) \varphi_{1,n}(v)\,dv$$
with $\varphi_{1,n}(v) =  \hv^{\otimes n} \hv \varphi(v)$. Recall that the notation $k^{\otimes n}::\hv^{\otimes n} = (k\cdot \hv)^n$ is simply for sake of presentation. Note that since $f_0(x,v)$ is compactly supported in $v$ and $\| V_{0,t}(x,v) -v\|_{L^\infty_{x,v}} \lesssim \epsilon$, we have $f_0(X_{0,t}(x,v) , V_{0,t}(x,v))$ vanishes for $|v|\ge 2R_0$ for some $R_0>0$. Therefore, for $A_0  = 2|R_0|/\langle 2R_0\rangle <1$, there exists some universal constant $C_0$, independent of $n$, so that \begin{equation}\label{bdvarphin}
|\varphi_{1,n}(v)| \le C_0 A_0^n , \qquad |\nabla_v \varphi_{1,n}(v)| \le C_0 nA_0^n, 
\end{equation} 
uniformly for $|v|\le 2R_0$, for all $n\ge 1$. Therefore, we may use \eqref{eq:fouriermult} with $\delta =2$ and the results in Proposition \ref{prop-bdS}, yielding 
$$ 
\begin{aligned}
\|S_{\pm,1}(t) \|_{L^\infty_{x} } &\le \sum_{n\ge 0} \| a_{\pm,n}(i\partial_x)  :: a_\pm(i\partial_x) \nabla_x S_{1,n}(t)\|_{L^\infty_x}
\\&\lesssim \sum_{n\ge 0}  \| S_{1,n}(t)\|_{L^\infty_x}
\lesssim \epsilon_0 \langle t\rangle^{-3} 
\sum_{n\ge 0}  A_0^n 
\\&\lesssim \epsilon_0 \langle t\rangle^{-3} .
\end{aligned}$$
Similarly, we bound 
$$ 
\begin{aligned}
\|S_{\pm,1}(t) \|_{H^{\alpha_0+1}_x} \le \sum_{n\ge 0}
\| a_{\pm,n}(i\partial_x) ::a_\pm(i\partial_x)  \nabla_x S_{1,n}(t)\|_{H^{\alpha_0+1}_x}
&\lesssim \sum_{n\ge 0}\| S_{1,n}(t)\|_{H^{\alpha_0}_x} \lesssim \epsilon_0 \langle t\rangle^{\delta_1},
\end{aligned}$$
which prove the desired bounds for the second term in \eqref{GSosc1}. 

\subsubsection*{Source term $S_{\pm,2}$.}

Finally, we bound the last convolution $G_\pm^{osc} \star_{t,x} \nabla_xS_{\pm,2}(t,x)$  in the representation \eqref{GSosc1}. Recalling the definition of $S_{\pm,2}$ in \eqref{def-Spm12} and the expansion \eqref{expand-phi1}, we may write 
 $$ 
\begin{aligned}
S_{\pm,2}(t,x)  &= \sum_{n\ge 0} a_{\pm,n}(i\partial_x) [E(t,x)S_{2,n}(t,x)]
 \end{aligned}$$ 
for the same coefficients $a_{\pm,n}(k)$ as in \eqref{seriesn}, recalling $a_{\pm,n}(k) \le \langle k\rangle^{-1}$, where 
$$ S_{2,n}(t,x) =  
 \int f_0(X_{0,t}(x,v) , V_{0,t}(x,v)) \varphi_{2,n}(v)\,dv$$
for $\varphi_{2,n}(v) = \nabla_v ( \hv^{\otimes n} \hv )\varphi(v).$  
Observe that $S_{2,n}(t,x)$ are again of the same form as that of the source density $S(t,x)$ defined as in \eqref{defS}. Therefore, we may write 
$$
  G_\pm^{osc} \star_{t,x} \nabla_xS_{\pm,2}(t,x) = 
   \sum_{n\ge 0} a_{\pm,n}(i\partial_x) :: G_\pm^{osc} \star_{t,x} \nabla_x[E(t,x)S_{2,n}(t,x)]
.$$
Using Lemma \ref{lem-decayosc}, it suffices to prove that 
\begin{equation}\label{checkES} \int_0^t \| ES_{2,n}(s)\|_{W_x^{3,1}} \;ds \lesssim 1, \qquad \|  ES_{2,n}(t)\|_{H^{1}_x} \lesssim \langle t\rangle^{-5/2} .\end{equation}
Indeed, using the decay estimates in Proposition \ref{prop-bdS}, we bound 
$$
\begin{aligned}
\| ES_{2,n}(t)\|_{H^1_x} 
& \lesssim \| E(t)\|_{L^\infty}\| S_{2,n}(t)\|_{H^1_x} + \| E(t)\|_{H^1_x}\| S_{2,n}(t)\|_{L^\infty}
\lesssim \epsilon^2 \langle t\rangle^{-3}.
\end{aligned}
$$
Similarly, we bound 
$$
\begin{aligned}
\| ES_{2,n}(t)\|_{W_x^{3,1}} 
& \lesssim \| E(t)\|_{H_x^{3}} \|S_{2,n}(t)\|_{H_x^{3}}   \lesssim \epsilon^2 \langle t\rangle^{-\frac32 + \delta_3}\end{aligned}
$$
for $\delta_3 = \frac{3}{|\alpha_0|-1}$. This verifies \eqref{checkES}, since $\delta_3 <\frac12$, provided that $|\alpha_0|\ge 8$. 
Therefore, we obtain 
$$ \| G_\pm^{osc} \star_{t,x} \nabla_xS_{\pm,2}(t)\|_{L^\infty_x} \lesssim \epsilon^2 \langle t\rangle^{-3/2} .$$
It remains to prove the estimates in $H^{\alpha_0+1}_x$. Indeed, using \eqref{Gosc-l2}, for each $n$, we bound 
$$ 
\begin{aligned}
\| a_{\pm,n}(i\partial_x) & G_\pm^{osc} \star_{t,x} \nabla_x[E(t,x)S_{2,n}(t,x)](t)\|_{H^{\alpha_0+1}_x} 
\\&\lesssim
  \int_0^t \| a_{\pm,n}(i\partial_x)[ES_{2,n}](s)\|_{H^{\alpha_0+1}_x} \; ds
\\
&\lesssim
 \int_0^t \Big[ \| E(s)\|_{L^\infty_x}\|S_{2,n}(s)\|_{H^{\alpha_0}_x} +  \| E(s)\|_{H^{\alpha_0}_x}\|S_{2,n}(s)\|_{L^\infty_x}\Big]\; ds
  \\&\lesssim \epsilon^2\int_0^t \langle s\rangle^{-3/2+\delta_1} \; ds 
    \\&\lesssim \epsilon^2,
  \end{aligned}$$
 noting that $\delta_1 < 1/2$, provided that $|\alpha_0|\ge 4$. The infinite series in $n$ converges absolutely as in the previous case, using \eqref{bdvarphin}. This proves the desired estimates for the last term in  \eqref{GSosc1}, and therefore completes the proof of Proposition \ref{prop-decayEosc}.

\appendix 

\section{Fourier multipliers}
\label{sec-FM}
We recall the homogeneous Littlewood-Paley decomposition on $\mathbb{R}^{d}$, $d \in \mathbb{N}$, which reads
\begin{equation}\label{def-LP}
h= \sum_{k \in \mathbb{Z}} P_k h, 
\end{equation}
where $P_k$ denotes the standard Littlewood-Paley projection on the dyadic interval $[2^{k-1}, 2^{k+1}]$. We note the following classical Bernstein inequalities (see, e.g.,
\cite{BCD})
\begin{equation}\label{Berstein}
\| P_k \partial_x h \|_{L^p} \lesssim 2^{k} \| h\|_{L^p}, \qquad  2^{k} \| P_kh\|_{L^p} \lesssim \| \partial_x h\|_{L^p_x}
\end{equation}
for all $p \in [1,\infty]$ and $k \in \mathbb{Z}$. In addition, we recall the following definition of Besov spaces.

\begin{defi}
Let  $s \in \mathbb{R}$, $p,r \in [1,+\infty]$. The Besov space $B^{s}_{p,r}$ is defined so that the following norm is finite
$$
\| h\|_{B^{s}_{p,r}} =  \left\| \sum_{k\leq 0} P_k h \right\|_{L^p} +\left ( \sum_{k\ge 0} 2^{s kr}\| P_k h\|^r_{L^p} \right)^{1/r}< +\infty  .
$$
\end{defi}

\begin{lem}
\label{lem:fouriermult}
 Fix $\delta>0$. 
Let $\sigma(\xi)$ be sufficiently smooth and satisfy 
$$|\partial^\alpha \sigma(\xi)| \leq C_\alpha \langle \xi \rangle^{-\delta-|\alpha|} $$ 
for $|\alpha|\ge 0$.
Then, there hold 

\begin{enumerate}

\item The Fourier multiplier $\sigma(\xi)$ is a bounded operator from $L^p$ to $L^p$ for $p \in [1,\infty]$. 

\item We have for all $p \in [1,\infty]$,
\begin{equation}
\label{eq:fouriermult}
\|  \sigma(i\partial) \nabla  f \|_{L^p} \lesssim  \| f \|_{L^p}^{\frac{\delta}{1+\delta}}  \| \nabla f \|_{L^p}^{\frac{1}{1+\delta}} \lesssim \| f\|_{W^{1,p}}.
\end{equation}

\item We have for all $p \in [1,+\infty]$, $r \in [1,\infty)$, and $s\ge 0$ so that $\delta > s - \lfloor s \rfloor$, 
\begin{equation}
\label{eq:fouriermultbesov}
\|  \sigma(i\partial) \nabla  f \|_{B^{s}_{p,r}} \lesssim   \| f \|_{L^p}^{\frac{\delta}{1+\delta}}  \| \nabla f \|_{L^p}^{\frac{1}{1+\delta}}+ \| \nabla^{1+\lfloor s \rfloor} f\|_{L^p} \lesssim \| f\|_{W^{1+\lfloor s \rfloor,p}},
\end{equation}
where $\lfloor s \rfloor$ denotes the largest integer that is smaller than $s$. 

\end{enumerate}
\end{lem}

\begin{proof} Using the Littlewood-Paley decomposition, we write
$$
\sigma(i\partial)f = P_{k\le 0}(  \sigma(i\partial) f)  +  \sum_{k \geq 0 } P_k( \sigma(i\partial)  f)
$$
in which $P_{k\le 0}(  \sigma(i\partial) f)$ is a convolution in space of a rapidly decaying function against $f$. The boundedness from $L^p$ to $L^p$ thus follows. On the other hand, using Berstein's inequalities, we bound   
$$ \| \sum_{k \geq 0 } P_k( \sigma(i\partial)  f)\|_{L^p}\le  \sum_{k \geq 0 } \|P_k( \sigma(i\partial)  f)\|_{L^p} \lesssim \sum_{k\ge 0} 2^{-k\delta} \| f\|_{L^p} \lesssim \| f\|_{L^p},$$
in which the summation is finite, since $\delta>0$. As for (2), again using the Littlewood-Paley decomposition, we write
$$
\sigma(i\partial)  \nabla  f = \sum_{k \le A } P_k(  \sigma(i\partial) \nabla  f)  +  \sum_{k \geq A } P_k( \sigma(i\partial)  \nabla  f) =: f_1 + f_2, 
$$
for $A\in \ZZ$ to be determined. 
Using \eqref{Berstein}, we first bound 
$$
\| f_1\|_{L^p} \lesssim \sum_{k <A } \| P_k( \sigma(i\partial) \nabla  f)\|_{L^p} \lesssim  \sum_{k <A } 2^k    \| f\|_{L^p}  \lesssim 2^A \| f\|_{L^p}. 
$$
Similarly, recalling $| \sigma(\xi)| \lesssim\langle \xi \rangle^{-\delta} $, we bound 
$$
\| f_2\|_{L^p} \lesssim \sum_{k \geq A } \| P_k( \sigma(i\partial) \nabla  f)\|_{L^p} 
\lesssim  \sum_{k \geq A } 2^{-\delta k}    \| \nabla f\|_{L^p}  \lesssim 2^{-\delta A}    \|\nabla f\|_{L^p}. 
$$
We can now fix $A$ such that $2^A \| f\|_{L^p}=2^{-\delta A}    \|\nabla f\|_{L^p}$, leading to \eqref{eq:fouriermult}. Finally, by definition and \eqref{Berstein}, we bound 
$$
\begin{aligned}
\|  \sigma(i\partial) \nabla  f \|_{B^{s}_{p,r}} &\lesssim  \|  \sigma(i\partial) \nabla  f \|_{L^p}  
+ \left( \sum_{k \geq 0} 2^{k sr} \| P_k (\sigma(i\partial) \nabla  f)\|_{L^p}^r\right)^{1/r} \\
&\lesssim   \| f \|_{L^p}^{\frac{\delta}{1+\delta}}  \| \nabla f \|_{L^p}^{\frac{1}{1+\delta}}+  \left( \sum_{k \geq 0} 2^{-k(\delta-s+\lfloor s \rfloor) r}\right)^{1/r} \|  \nabla^{1+\lfloor s \rfloor}  f\|_{L^p},
\end{aligned}
$$
yielding \eqref{eq:fouriermultbesov}. The lemma follows. 
\end{proof}

\begin{remark}
Following the proof of Lemma \ref{lem:fouriermult}, we obtain the following standard interpolation inequalities
  \begin{equation}\label{dx-interpolate}
\| \partial^\beta_x f\|_{L^\infty}  \lesssim \| f\|_{L^\infty}^{1-|\beta|/|\alpha|} \| \partial_x^\alpha f\|_{L^\infty}^{|\beta|/|\alpha|} \lesssim \| f\|_{L^\infty}^{1-|\beta|/|\alpha|} \| f\|_{H^{2+\alpha}}^{|\beta|/|\alpha|} 
\end{equation}
for $0 \le |\beta| \le |\alpha|$. 
\end{remark}

\bibliographystyle{abbrv}

\end{document}